\documentclass[12pt]{amsart}
\usepackage{amsmath,amssymb,amsfonts,latexsym}
\usepackage{mathrsfs}%pour les caracteres en calligraphie
\usepackage[dvips]{geometry}
\usepackage{array}
\usepackage{epsf}
\usepackage{epsfig}
\newtheorem*{lemma}{Lemma}
\newtheorem*{prop}{Proposition}
\newtheorem*{thm}{Theorem}
\newtheorem*{cor}{Corollary}
\newcommand{\iso}{\overset{\sim}{\rightarrow}}

\newcommand{\twoheaddownarrow}{\overset{\sim}{\twoheaddownarrow}}

\newcommand{\nc}{\newcommand}
\nc{\Ker}{\operatorname{Ker}} \nc{\rker}{\operatorname{rKer}}
\nc{\im}{\operatorname{Im}}
\nc{\stab}{\operatorname {Stab}}
\nc{\ann}{\operatorname {Ann}}
\nc{\Id}{\operatorname {Id}}
\nc{\Prim}{\operatorname {Prim}}
\nc{\Real}{\operatorname {Re}}
\nc{\Ext}{\operatorname {Ext}}
\nc{\rad}{\operatorname {rad}}
\nc{\sn}{\operatorname {sn}}
\nc{\wt}{\operatorname {wt}}
\nc{\Max}{\operatorname {Max}}
\nc{\Supp}{\operatorname {Supp}}
\begin{document}
\title[Preparation $B(\infty)$]{A Preparation Theorem for the Kashiwara $B(\infty)$ Crystal}

%\title[The $B(\infty)$ Crystal]{Dual Kashiwara Functions for the $B(\infty)$ Crstal}
\author[Anthony Joseph]{Anthony Joseph}

\date{\today}
\footnote{Work supported in part by the Binational Science Foundation, Grant no. 711628}
\maketitle

\vspace{-.9cm}\begin{center}
Donald Frey Professional Chair\\
Department of Mathematics\\
The Weizmann Institute of Science\\
Rehovot, 76100, Israel\\
anthony.joseph@weizmann.ac.il
\end{center}\

\medskip


%\title[Preparation $B(\infty)$]{A Preparation Theorem for the Kashiwara $B(\infty)$ Crystal}
%
%\authors{Anthony Joseph \address
%Department of Mathematics
%\\
%Weizmann Institute of Science, \\ Rehovot, 76100,
%Israel \email anthony.joseph@weizmann.ac.il}

 %\footnotetext[1]{Work supported in part by
%\mathbf{}Israel Science Foundation Grant, no. 797/14.}
%

\date{\today}
\maketitle

\bigskip

\medskip


\

\textbf{Abstract}.  The Kashiwara $B(\infty)$ crystal pertains to a Verma module for a Kac-Moody Lie algebra. Ostensibly it provides only a parametrisation of the global/canonical basis for the latter. Yet it is much more having a rich combinatorial structure from which one may read off a parametrisation of the corresponding basis for any integrable highest weight module, describe the decomposition of the tensor products of highest weight modules, the Demazure submodules of integrable highest weight modules and Demazure flags for translates of Demazure modules.

$B(\infty)$ has in general infinitely many presentations as subsets of countably many copies of the natural numbers each given by successive reduced decompositions of Weyl group elements. In each presentation there is an action of  Kashiwara operators determined by Kashiwara functions.  These functions are \textit{linear} in the entries.  Thus a natural question is to show that in each presentation the subset $B(\infty)$ is polyhedral.   Here a new approach to this question is initiated based on constructing dual Kashiwara functions and in this it is enough to show that the latter are also linear in the entries.

The basic hypothesis is that dual Kashiwara functions exist and can be expressed as differences of successive Kashiwara functions with non-negative integer coefficients.  This last requirement is called the positivity condition. Here one starts from an explicit linear function called the ``initial driving function''.   Then in terms of the reduced decomposition one seeks an algorithm to provide further dual Kashiwara functions which in particular will provide ``driving functions'' for the next induction step.

The present work resolves one of the two very difficult obstacles in this step-wise construction, namely that the resulting functions must satisfy a sum, or simply $S$, condition.  It depends very subtly on inequalities between the coefficients occurring in the driving function obtained from the previous step.  The only remaining obstacle, that sufficiently many functions are obtained, can at least be verified in many families of cases, though this is to be postponed to a subsequent paper.

The set $H$ of functions constructed here are described through equivalence classes of tableaux with boundary conditions and whose entries are the coefficients $\textbf{c}:=\{c_i\}_{i=1}^t$ provided by a given driving function. The latter are positive integers expressible in terms of entries in the Cartan matrix. The set $H$ admits an involution called duality. Then a subset $H(\textbf{c})$ of $H$ is constructed which satisfies the positivity condition, and compatibility with duality. It depends on the $t!$ possible orderings defined by inequalities between these coefficients.  Then $H(\textbf{c})$ is further subdivided into subsets $H_j(\textbf{c})$ of functions vanishing on the relative $j^{th}$ entry of $B(\infty)$ corresponding to the root $\alpha$.

The Preparation Theorem can be roughly described as follows.  If the Kashiwara operators defined by $\alpha$ enter at the $j^{th}$ entry, then there is an element $f \in H_j(\textbf{c})$ which dominates a given element of $H(\textbf{c})$.  This can be reformulated as an ``$S$-condition'' which is independent of any knowledge of $B(\infty)$. Its proof involves a delicate interplay between the inequalities satisfied by the coefficients $\textbf{c}$ and the expressions for the elements of $H(\textbf{c})$.

The proof of the Preparation Theorem starts from a description of the graph $\mathscr G_H$ of links whose vertices are elements of $H$ and whose edges are ``one block linkages'' of tableaux.  The inequalities between the elements of $\textbf{c}$ are encoded through ``triads'' in these graphs.  Then certain $S$-graphs $\mathscr G(\textbf{c})$ with triads compatible with the given linear order on $\textbf{c}$ and which satisfy a graphical version of the $S$ condition, are constructed.  Finally  $\mathscr G(\textbf{c})$ is identified with the subgraph $\mathscr G_{H(\textbf{c})}$ of $\mathscr G_H$ defined by $H(\textbf{c})$.

This theory has some intriguing numerology.  One finds that $|H(\textbf{c})|=2^t$, whilst $|H|=C(t+1)$, where $C(t)$ is the $t^{th}$ Catalan number.  Again $|H_j(\textbf{c}|$ is a power of $2$ depending on the choice of $j$ \textit{and} the inequalities between the elements of $\textbf{c}$, excluding thereby the possibility of any simplification coming from an action of the symmetric group $S_t$. On the other hand $|H_j|=C(j-1)C(t-j+1)$, whilst for all $j$, the common intersection of the $H_j(\textbf{c})$ over all possible choices of inequalities has cardinality $1$.  Again because of coincidences the number of distinct graphs $\mathscr G(\textbf{c})$ is generally less than $t!$ and indeed quite remarkably equal to $C(t)$.

\section{Introduction}\label{1}

\subsection{}\label{1.1}

Let $U_q(\mathfrak g)$ be the quantized enveloping algebra of a Kac-Moody algebra $\mathfrak g$ corresponding to a symmetrizable Cartan matrix. With respect to a choice of a Cartan subalgebra $\mathfrak h$ of $\mathfrak g$ and $\pi$ a choice of simple roots, let $U^+$ be the subalgebra of $U_q(\mathfrak g)$ generated by the simple root vectors. Let $W$ denote the (Weyl) group generated by the simple reflections acting on $\mathfrak h^*$.  A highest weight vector of a $U_q(\mathfrak g)$ module is one annihilated by the augmentation ideal of $U^+$ which is also an eigenvector for the action of the torus $T$ whose eigenvalue, called its weight, can be identified with an element $\lambda \in \mathfrak h^*$.  A highest weight module is one generated by a highest weight vector.

Partly inspired by the work of Lusztig \cite {Lu1}, Kashiwara \cite {Ka1} studied a $q \rightarrow 0$ limit of a highest weight module $V$ of $U_q(\mathfrak g)$ obtaining thereby what he called a crystal $B$ or crystal graph.  The vertices of this graph correspond to a ``global" basis of $V$ consisting in particular of weight vectors obtained by ``reversing" the $q \rightarrow 0$ limit.  The edges of $B$ labelled by $\alpha \in \pi$ are given Kashiwara operators $e_\alpha, f_\alpha$ which partly recover the action of the simple root vectors on the global basis.

  In addition each element $b \in B_J$ is assigned a weight $\wt b$ with the rule that $\wt (e_\alpha b)=\wt b +\alpha$. Then the formal character of B is defined to be the sum $\sum_{b\in B_J} e^{\wt b}$.  If $V$ is a Verma module of highest weight $0$, Kashiwara's construction gives what we call the Kashiwara $B(\infty)$ crystal. It turns out to have some quite remarkable properties some which are noted in the subsection below.

\subsection{}\label{1.2}

First the global basis coincides with Lusztig's canonical basis \cite {GL}.

Secondly $B(\infty)$ may be considered as just a combinatorial object.  Yet from this combinatorial structure one may deduce for any dominant integral weight $\lambda$ the structure of the crystal $B(\lambda)$ corresponding to the integrable highest weight module $V(\lambda)$ of highest weight $\lambda$.

 For all $w\in W$ let $V_w(\lambda)$ denote the $U^+$ module generated by the unique up to scalars weight vector of weight $w\lambda$.  It is called the Demazure module of extremal weight $w\lambda$.  Let $M$ be $TU^+$ module. A descending sequence of submodules of $M$ whose successive quotients are Demazure modules is called a Demazure flag for $M$.

Kashiwara \cite {Ka2} showed that a sub-basis of the global basis of $V(\lambda)$ is basis of $V_w(\lambda)$.  Let  $B_w(\lambda)$ be the corresponding subset of $B(\lambda)$.  It is not quite a sub-crystal.  Yet one may give a combinatorial version of a Demazure flag for certain tensor products  \cite {J4}, \cite {L3}.  Again Kashiwara \cite {Ka2} defined a $\lambda \rightarrow \infty$ limit $B_w(\infty)$ of $B_w(\lambda)$.  When $W$ is finite and $w_0$ its unique longest element one has $B_{w_0}(\infty)=B(\infty)$.

Otherwise we view $B(\infty)$ as a direct limit of the $B_w(\infty): w \in W$.

Perhaps the most remarkable result of Kashiwara \cite {Ka1} especially for our purposes is the existence of an involution $\star$ on $B(\infty)$. It is highly non-trivial, even though it comes from a rather trivial antiautomorphism of $U^+$. Since $B(\lambda)$ can be viewed as a subset of $B(\infty)$,  we can view $B(\lambda)^\star$ also as a subset of $B(\lambda)$, generally different to $B(\lambda)$.

Let $\mu, \lambda$ be dominant integral weights. Then by \cite {J4} the weights of the intersection $B(\mu)\cap B(\lambda)^\star$ are just the highest weights of the irreducible components of the tensor product $V(\mu)\otimes V(\lambda)$.

In view of the above it is not surprising that the combinatorial structure of $B(\infty)$ is extremely complicated even though the Verma module itself is rather simple.  This is most starkly illustrated by the question of computing  the formal character of $B(\infty)$.  Since $B(\infty)$ is the parametrization of the basis of a Verma module (of highest weight $0$) it must have the same formal character as the latter, which of course is rather easy to compute.  Yet it is still not known how to do this purely combinatorially except if $W$ is finite.  One does have a formula for the formal character of $B_w(\lambda)$ through Demazure operators, but the resulting formula for $B_w(\infty)$ by taking limits is not very transparent and gives little hint as to the formal character of the direct limit.

  By contrast if $W$ is finite, then $B_{w_0}( \lambda)$ has a simple formula because its formal character must be $W$ stable (see \cite [Cor. 6.3.16]{J3} - the argument is due to Demazure). From this we can take a $\lambda \rightarrow \infty$ limit to obtain the formal character of $B(\infty)$ .  We can therefore anticipate a much greater difficulty in studying $B(\infty)$ when $W$ is infinite.

\subsection{}\label{1.3}

The observations of \ref {1.2} motivate the search for a precise description of $B(\infty)$. In this let us recall its construction due to Kashiwara \cite {Ka2}.

Let $J$ be a sequence of simple roots in which every simple root occurs sufficiently many times.  In this it is enough to take $J_w:w\in W$ to be the sequence defined a reduced decomposition of $w$ and to view $J$ as an appropriate limit of the $J_w: w \in W$.  In particular if $W$ is finite we may choose $J$ to be some $J_{w_0}$. In the latter case $|J|$ is just the number of positive roots of $\mathfrak g$. In general $J_w$ constructs $B_w(\infty)$.

Kashiwara \cite {Ka1} defined a crystal $B_J$ whose elements lie in the set $\mathbb N^{|J|}$ of all sequences of length $\leq |J|$ of natural numbers with only finitely many non-zero entries.  The rules for applying $e_\alpha, f_\alpha$ to $B_J$ are given in terms of what we call the Kashiwara functions which notably are \textit{linear} in the entries.

$B_J$ admits a canonical element $b_\infty$ corresponding to all entries being $0$.  Let $B_J(\infty)$  be the subset of $B_J$ generated by the action of $f_\alpha: \alpha \in \pi$ on $b_\infty$.  A remarkable result of Kashiwara \cite {Ka2} is that $B_J(\infty)$ is stable under the action of the $e_\alpha:\alpha \in \pi$.

As a crystal $B_J$ depends on the choice of $J$.   More specifically when $W$ is finite, $B_J$ depends on the choice of a reduced decomposition of $w_0$.  Again $B_J(\infty)$ as a subset of $\mathbb N^{|J|}$ also depends on $J$.  A remarkable result of Kashiwara \cite {Ka2} is that $B_J(\infty)$  as a \textit{crystal} is independent of $J$.  This crystal is $B(\infty)$.

An intriguing question arising from this construction is whether $B_J(\infty)$ is a sub-semigroup of $\mathbb N^{|J|}$.  Of course this would mean that there can be many different semi-group structures on $B(\infty)$.  Optimistically at least one of these semi-group structures would be free and then the weights of the generators would be the negative roots of $\mathfrak g$.  Unfortunately this generally fails though it does hold for $\mathfrak g$ simple of type $A$.

 \subsection{}\label{1.4}

 A natural way to settle the question raised in \ref {1.3} is to show that $B_J(\infty)$ is a polyhedral subset of $B_J$, that is to say given by a family of \textit{linear} inequalities on the entries. Ultimately we would like to compute the linear functions involved.

 An intriguing suggestion as how to achieve the above result was put forward by Nakashima and Zelevinsky \cite {NZ} using differences of successive Kashiwara functions.  Yet they could only establish the required result up to a conjecture which is false in almost all cases.  Roughly speaking they were rather indiscriminate as to which differences should appear.

 For $\mathfrak g$ simple of type $A$ a beautiful solution to the construction of the desired inequalities was given by Gleizer and Postnikov \cite {GP} using wiring diagrams.  One can interpret their construction as specifying \textit{ exactly} which differences should appear.  However it seems practically impossible to understand how this should generalize to other cases.

 Indeed Berenstein and Zelevinski \cite {BZ} showed that $B(\infty)$ is polyhedral when $W$ is finite by a completely different method.  In this the inequalities obtained are less transparent.

 \subsection{}\label{1.5}

 Here we initiate a different approach to showing that $B(\infty)$ is polyhedral through the use of the Kashiwara involution.  It is rather easy to show that if what we call the ``dual Kashiwara functions" \textit{exist} (one for each $\alpha \in \pi$) and are linear, then $B(\infty)$ is polyhedral.  As might be expected from the above discussion these dual Kashiwara functions should be expressed as differences of successive Kashiwara functions; but that their exact choice will be a\textit{ very delicate matter.}  That such a choice is possible is the essence of the Preparation Theorem.

 \subsection{}\label{1.6}

  When $W$ is finite it is possible to relate the different presentations of $B_{w_0}(\infty)$ obtained as above through the various reduced decompositions of $w_0$.  One combines the results of Kashiwara \cite {Ka2} in rank $2$ with the machinery of tropical semi-fields which gives the required result in terms of rational functions. In this addition is presented as taking a maximum, whilst multiplication and division are addition and substraction respectively. From this one may \textit{in principle} compute dual Kashiwara functions.  The only snag is that we need to know that of these rational functions the ones we need are Laurent polynomials.  Nevertheless this method is a useful computation tool as computer calculations can check Laurent polynomiality.

  %In the general case one may compute dual Kashiwara functions inductively; but even in the simplest cases this takes tens of pages of computation.  The difficulty is inherent in the enormous complexity of the Kashiwara involution.

  Again in the case of type $A$ wiring diagrams can be used  to obtain dual Kashiwara functions.  However there are some difficulties extending this theory. First of all in this method only defines dual Kashiwara functions on $B_J(\infty)$, and not in $B_J$ as in the case of the Kashiwara functions themselves.  Again the solution obtained in type $A$ is obtained by adjoining functions associated to combining ``faces"  defined by the wiring diagram.  Outside type $A$ one may need to use the same face several times and it is quite unclear how many times.

 Our method to compute the dual Kashiwara functions does not depend on the finiteness of $W$.  It uses the fact due to Kashiwara \cite {Ka2}, that the dual Kashiwara operators almost commute with the Kashiwara operators themselves. This implies that the dual Kashiwara functions must have maxima which are ``almost invariant'' on $B_J(\infty)$.  We define them on $B_J$ by extending this almost invariance by induction on reduced decomposition.  In turns out that except in the first step which provides the ``initial driving function'' we can assume ``invariance" always holds.  We remark that for type $A$ this driving function can be interpreted as the unique open face (for the corresponding $\alpha \in \pi$) in the foot of the wiring diagram.
 \subsection{}\label{1.7}

 The construction of a set of linear functions whose maxima admit ``invariance", forms the subject of the present paper.  It leads to problem which is of consuming interest in its own right.

 In this the problem can be formulated without any reference to crystals and the notion of ``invariance".

 At each step of the reduced decomposition defined by an element $\alpha \in \pi$, we start from a driving function $h$ obtained from the previous step.  It depends on a set of $\textbf{c}=c_1,c_2,\ldots,c_t$ of $t$ natural numbers and $t+1$ indeterminates $m_1,m_2,\ldots, m_{t+1}$ representing entries in $B_J$ corresponding to $\alpha$.  Then we construct a set
of functions $H^{t+1}$, or simply $H$, with subsets $H^{t+1}(\textbf{c})$ defined by the relative inequalities between the $\{c_i\}_{i=1}^t$. The set $H$ admits a further decomposition  $H=\sqcup_{j=1}^{t+1}H_j$, where $H_j$  denotes the subset of functions which have a vanishing coefficient of $m_j$.  The elements of these sets are defined by tableaux in a manner analogous but quite different to the derivation of a Specht basis from Young tableaux.

 The Preparation Theorem (Theorem \ref {8.6}) states that for each $j = 1,2,\ldots, t+1$ such that the Kashiwara operators enter at the $j^{th}$ place there must be a function $f \in H_j(\textbf{c})$ dominating a given function in $H(\textbf{c})$.  The significance of this result concerning invariance is explained in \ref {6.7}.

 \subsection{}\label{1.8}

 The difficulty in proving the Preparation Theorem lies in the fact that two distinct sets of inequalities must be adroitly balanced.
 %It seems almost impossible to show this just using the functions or their associated tableaux.
 Its resolution requires that we study some associated graphs.

 We define graphs $\mathscr G_{H^{t+1}}: t \in \mathbb N^+$ whose vertices are elements of $H^{t+1}$ and whose edges form  ``single block linkages" of tableaux labelled by the $c_i$.  If the element in question belongs to $H_j$, for some $j \in \{1,2,\ldots,t+1\}$, then the corresponding vertex carries the label $j$.  This is analogous but different to Gerstenhaber's well-known graph of inclusions of nilpotent orbits in type $A$.

 A linear order on $\textbf{c}$ is encoded in the graphs through ``triads". Then an ``ordered path'' is a sequence of vertices for which the indices over successive edges increase.   A graph is called an $S$-graph if for every edge $v'$ and every $j \in \{1,2,\ldots,t+1\}$, there exists an edge $v$ and an ordered path from $v'$ to $v$.  The Preparation Theorem is equivalent to showing that $\mathscr G_{H(\textbf{c})}$ is an $S$-graph.

 We provide in Section \ref {7} a way to construct an $S$-graph $\mathscr G(\textbf{c})$ by induction on $t$ admitting triads which are compatible with the linear order on $\textbf{c}$.  We show in Section \ref {8} that $\mathscr G(\textbf{c})$ is isomorphic to $\mathscr G_{H(\textbf{c})}$.  This leads to the Preparation Theorem (Theorem \ref {8.6}).  We conclude with a direct proof of the Preparation Theorem for tableaux of height one (\ref {8.7}). Even in this simple case the proof is not easy which we believe indicates the depth behind the $S$-graph construction.

\subsection{}\label{1.9}

A remarkable consequence of our construction is that $|H^{t+1}(\textbf{c})|=2^t$.  Moreover every $|H_j(\textbf{c})|$ is a power of $2$ dependent on $j$  which can be calculated from the given linear order \cite {JL2}.  For each value of $j$ these powers depend on $\textbf{c}$ in a manner which shows that there can be no reasonable action of the symmetric group $S_t$ permuting the $H(\textbf{c})$.

A further remarkable consequence of our construction is that the union of the graphs $\mathscr G_{H(\textbf{c})}$, namely $\mathscr G_{H^{t+1}}$, has as its number of vertices the Catalan number $C(t+1)$. We call a set with cardinality
$C(t)=\frac{(2t)!}{t!(t+1)!}$, a Catalan set of order $t$.  At least $66$ such sets are known \cite [Exercises] {St}, from which one may surmise that they are of considerable interest to combinatorists.  An example which may interest Lie algebraists is the set $\mathscr I_t$ of ideals of $t\times t$ strictly upper triangular matrices stable under conjugation by the diagonal.

%Set $C_t=\frac{(2t)!}{t!(t+1)!}$.  It is an integer known as the $t^{th}$ Catalan number.  Call $S$ a Catalan set of order $t$ if $|S| = C_t$.  Many different examples of Catalan sets are known. Most of these come with a symmetry, for example the set $\mathscr I_t$ of ideals of $t\times t$ strictly upper triangular matrices stable under conjugation by the diagonal.

In a subsequent paper \cite {JL2} we show that $H^{t+1}$ is a Catalan set of order $t+1$ and further calculate $|H^{t+1}_j|$.  Moreover $H^{t+1}$ lacks the symmetry which the known Catalan sets often exhibit and this lack of symmetry is quite fundamental being related to the lack of symmetry of $e_\alpha$ (or $f_\alpha$) insertions in $B_J$.  Again $H^{t+1}$ comes with two sets of labellings which translate to two sets of labellings on $\mathscr G_{H^{t+1}}$.   In particular this labelled graph structure is shown \cite {JL2} to canonically determine the $S$-subgraphs $\mathscr G(\textbf{c})$.  Even dropping the labelling, $\mathscr G_{H^{t+1}}$ is quite different to say the graph $\mathscr G(\mathscr I_t)$ of inclusions in $\mathscr I_t$. Finally whereas the number of $S$ graphs is $t!$, they are not all distinct and the number of distinct $S$-graphs is again a Catalan number, namely $C(t)$ \cite {JL2}.
%the decomposition of $H^{t+1}$ into the $C(t)$ subsets $\mathscr G(\textbf{c})$ is \textit{canonical} being determined by the structure of the labelled graph $\mathscr G_{H^{t+1}}$.

All this seems to be new.

\subsection{}\label{1.10}

In the above it is not necessary to assume that the Cartan matrix of $\mathfrak g$ is symmetrizable except in so far as canonical/global bases are concerned.  Thus one may replace the Kashiwara theory of crystals by Littelmann's realization through piecewise linear paths.  This does not use $U_q(\mathfrak g)$ and so does not require symmetrizability.  In \cite {J5} it is shown how to define $B(\infty)$ in this framework and define an involution on it. In this manner the theory becomes purely combinatorial. Through \cite {JL}, \cite {L} it may be further possible to extend these considerations to Borcherds modification of the Cartan matrix.

\

\textbf{Acknowledgements.}   I would like to thank S. Zelikson for pointing out to me that one may compute the dual Kashiwara functions on $B_J(\infty)$ through the techniques of Gleizer-Postnikov \cite {GP} describing $B_J(\infty)$ as a polyhedral subset of $B_J$. This was an important starting point and it is unlikely that this paper could have got off the ground without it even though it is a much simpler case in which all the coefficients $\{c_i\}_{i=1}^t$ are equal and only $t+1$ functions (rather than $2^t$) are needed at each step. (It will be shown in a subsequent paper that the functions obtained  by Zelikson from \cite {GP} may be extended to $B_J$. This was also an important step as one can hardly hope to define functions on the unknown subset $B_J(\infty)$ of $B_J$ in general.) In addition Zelikson's computer computations for the dual Kashiwara functions in type $D_5$ using tropical semi-fields showed that for certain reduced decompositions in type $D_5$ one can have $t=3$ with the coefficients neither increasing nor decreasing.  Finally Zelikson pointed out to me that $|H^{t+1}|=C(t+1)$, for $t\leq 4$ and suggested that this should hold in general.

I should like to thank P. Lamprou for some computations which showed that one must go beyond diagrams of height $2$, as it eventually turned out already for $t=3$. Again her computations for $t=4$ extending mine for $t=3$ suggested the construction of $S$-graphs.  I originally gave a proof that these graphs satisfied an evaluation condition (see Corollary \ref {8.3}) but she was able to give a direct proof reproduced with her permission in Section \ref {7.6}.  I would like also to thank her for computations in the subsequent paper \cite {JL2} which intermesh with the present work and to thank both Lamprou and Zelikson concerning \cite {J6} which is to apply the present results and notably was their inspiration.

 I would like to thank V. Hinich for discussions concerning \ref {2.3.1}.

\section{Diagrams}\label{2}

\subsection{Unordered Partitions}\label{2.1}

Let $t$ be a non-negative integer and set $T:=\{1,2,\ldots,t\}$, $\hat{T}:= \{1,2,\ldots,t+1\}$.  A diagram $\mathscr D$ of order $t+1$ is a presentation of an unordered partition of a non-negative integer into $\leq t+1$ parts.

More precisely $\mathscr D$
is a collection of $t+1$ columns $C_1,C_2,\ldots,C_{t+1}$ placed on the $x$-axis with the subscript denoting the $x$ co-ordinate of the corresponding column. We view the columns as a stack of square blocks of equal size with the lowest block of each column (if it is not empty) having the same $y$ co-ordinate.  The rows $R_i:i=1,2,\ldots,$ are indexed by their $y$ co-ordinate. The height of $C_i$ is defined to be the number of blocks it contains and denoted by $ht C_i$. Then $\mathscr D$ is just the presentation of the unordered partition $(htC_1,htC_2,\ldots,htC_{t+1})$.  Since this paper is already long we shall use this device to desist from drawing diagrams (and tableaux) perhaps to the chagrin of some readers.

Define the height function $ht_{\mathscr D}$ on $\mathscr D$ by $ht_{\mathscr D}(i)=ht C_i$ and let $ht \mathscr D $ denote its maximal value.

\subsection{Boundary Conditions and the $*$ Operation}\label{2.2}

Let $s$ be a positive integer. Distinct columns $C,C'$ having height $\geq s$ of a diagram $\mathscr D$ of height $s$ are said to be neighbouring at level $s$ if every column of $\mathscr D$ between $C,C'$ has height $<s$.  A left (resp. right) extremal column $C$ of $\mathscr D$ is one which has no neighbours to the left (resp. right) at level $ht C$.

\textit{Boundary Conditions.}  The left (resp. right) boundary condition on $\mathscr D$ is that the height of any extremal left (resp. right) column of $\mathscr D$ must be either even (resp. odd) or the height of $\mathscr D$.

From now on all our diagrams will satisfy the boundary conditions, meaning both the left and right boundary conditions. Notice that this means that $C_{t+1}$ cannot be empty except if $\mathscr D$ is empty.

Given a diagram $\mathscr D$ we may obtain a new diagram $\mathscr D^*$ by interchanging $C_i$ with $C_{t+2-i}:i \in \hat{T}$ and increasing the height of every column by one.  We call the map $\mathscr D \mapsto \mathscr D^*$ the $*$ operation.  Eventually we shall introduce an equivalence relation on diagrams, so that it becomes involutive.  Then it will be called duality.

\subsection{Equivalence Classes}\label{2.3}

\subsubsection{Dominoes}\label{2.3.1}

Let $i$ a non-negative integer. A domino $D_{i+1}^{i+2}$, or simply $D$ if $i$ is unspecified, is a pair of blocks lying in rows $R_{i+1},R_{i+2}$ and the same column.

Consider a domino $D_{i+1}^{i+2}$ being adjoined to a column $C$ of height $i$ in a diagram $\mathscr D$ to give a new column $C \sqcup D$ of height $i+2$.  Then $D_{i+1}^{i+2}$ is called an even (resp. odd) domino if $i$ is even (resp. odd) and called a left (resp. right) domino if $C \sqcup D$ is the left (resp. right) neighbour at levels $i+1,i+2$ of a column $C'$ of height $\geq i+2$ in the new diagram.

The diagram $\mathscr D \sqcup D$ obtained by adjoining a left even or right odd domino $D$ to a diagram $\mathscr D$ again satisfies the boundary conditions. We call the map $\mathscr D \rightarrow \mathscr D \sqcup D$, domino adjunction. It commutes with the $*$ operation in the following sense.  If $D$ can be adjoined to $C_i$ in $\mathscr D$ to give $\mathscr D \sqcup D$, then $D$ can be adjoined to $C_{t+2-i}$ in $\mathscr D^*$ to give $\mathscr D^* \sqcup D$ and this is equal to $(\mathscr D \sqcup D)^*$.

Starting from a diagram $\mathscr D$ we obtain a graph $\mathscr G_{\mathscr D}$ whose vertices are diagrams with directed edges given by domino adjunction.  It is connected, directed and graded (by number of boxes). We call it the weak equivalence graph of $\mathscr D$. Its vertices form the weak equivalence class of $\mathscr D$ with respect to domino adjunction.

Since adjoining dominoes does not increase height, a diagram in the equivalence class of $\mathscr D$ has at most $(t+1)ht \mathscr D$ blocks.  Thus the equivalence graph of $\mathscr D$ is finite.

A diagram $\mathscr D$ is said to be complete (resp. deplete) if no left even or right odd dominoes can be adjoined (resp. removed) from $\mathscr D$.
%Thus a diagram is complete (resp. deplete) if it has the maximum (resp. minimum) number of blocks in its weak equivalence class.

\begin {lemma} A complete (resp. deplete) diagram in its weak equivalence class of $\mathscr D$ is unique.
\end {lemma}

\begin {proof}   Suppose we can adjoin a vertical domino $D$ to a column $C_i$ of $\mathscr D$ to obtain a new column $C_i\sqcup D$ of height $s+2$ in $\mathscr D \sqcup D$.  This means that $s$ is either even (resp. or odd) and $C_i\sqcup D$ has a right (resp. left) neighbour $C^r$ (resp. $C^\ell$) of height $\geq s+2$ at levels $s+1,s+2$.

Via the $*$ operation we can assume that $s$ is even.

Now suppose we can adjoin a vertical domino $D'$ to a column $C_j$ of $\mathscr D$ to obtain a column $C_j\sqcup D'$ in a new diagram $\mathscr D \sqcup D'$.  We claim that we can still adjoin $D$ to the column $C_i$ of $\mathscr D \sqcup D'$.

This is obvious if $i=j$.  If $i\neq j$ and we cannot adjoin $D$ to $C_i$ in $\mathscr D \sqcup D'$, then $C_j\sqcup D'$ must have height $s+1$ and lie strictly between $C_i$ and $C^r$.  Yet we can adjoin $D'$ to $C_j$ in  $\mathscr D$, so $C_j$ must have a left neighbour $C$ at level $s-1$ of height $\geq s+1$.  Again $C_i$ has height $s$, so $C$ lies strictly to the right of $C_i$ and strictly to the left of $C^r$,  contradicting that we can adjoin $D$ to $C_i$ in $\mathscr D$.

By the same argument we can adjoin $D'$ to $C_j$ in $\mathscr D \sqcup D$.  Thus $D,D'$ can be adjoined in either order to $\mathscr D$.  Moreover the result is obviously the same.

This means that in the weak equivalence graph $\mathscr G_{\mathscr D}$, any two outgoing edges from a vertex can be completed to a commuting square. A similar assertion holds for any two incoming edges to a vertex.  Finally recall that $\mathscr G_{\mathscr D}$ is connected and graded.  Hence the assertions (see Remark 1).

\end {proof}

\textbf{Remark 1.}  In the above we have used the following result which is surely well-known, but for lack of a reference we give details.

Let $\mathscr G$ be a connected, directed graph whose set $V(\mathscr G)$, or simply, $V$ of vertices is finite.  Assume that $\mathscr G$ is graded by a map $f: V \rightarrow \mathbb N^+$ with the property that $f(v)=f(v')+1$ if there is a directed edge from the vertex $v'$ to the vertex $v$ in $\mathscr G$.  Suppose further that given any three vertices $v_1,v_2,v_3 \in V$ and directed edges $v_1\rightarrow v_2,v_1\rightarrow v_3$, there exists $v_4 \in V$ and directed edges $v_2\rightarrow v_4, v_3\rightarrow v_4$.

Call $v \in V$ maximal if there is no $v' \in V$ with a directed edge $v \rightarrow v'$.

\begin {lemma}  $\mathscr G$ admits a unique maximal vertex.
\end {lemma}

\begin {proof}  Existence results from the finiteness of $V$.  Let $v,v'$ be distinct maximal vertices.  Since $\mathscr G$ is connected there is a path $p:(v'=v_1 \rightarrow v_2 \rightarrow \ldots \rightarrow v_n = v)$ from $v'$ to $v$ in $\mathscr G$.  We deduce a contradiction by considering the minimal value $m(p)$ taken by $f$ on the vertices of  a given path $p$.   Since $m(p)$ is bounded above by the maximum of $f(v'),f(v)$, we can assume that $m(p)$ takes it maximal value for all such paths.

Choose $i=1,2,\ldots,n$ such that $f(v_i)=m(p)$.  One cannot have $i=1,n$ since $v_1,v_n$ are maximal vertices.  Then by the hypothesis on $\mathscr G$ we can replace $v_i$ by $v'_i \in V$ having edges $v_{i-1}\rightarrow v_i', v_{i+1}\rightarrow  v_i'$.  Repeat for each such vertex $v_i$.  Then we have a path $p'$ from $v'$ to $v$ in $\mathscr G$ with $m(p')=m(p)+1$.  This gives the required contradiction.

\end {proof}

\textbf{Remark 2.}  Take $\mathscr D:=(3,2,1,4)$.  It is a depleted diagram.  We may adjoin a left even domino to $C_2$ to give $(3,2,3,4)$ and then a right odd domino to $C_3$ to give $(3,2,3,4)$ which is a complete diagram.  However these operations cannot be carried out in the opposite order.  The resulting weak equivalence graph is linear (in this case).

\subsubsection{}\label{2.3.2}

Let $\mathscr D$ be a complete diagram of height $r$.

\begin {lemma}

\

(i) $ht_\mathscr D C_1 \geq r-1$ and $ht_\mathscr D C_1=r$ if $r$ is even.

\

(ii) $ht_\mathscr D C_{t+1}\geq r-1$ and $ht_\mathscr D C_{t+1}=r$ if $r$ is odd.

\end {lemma}

\begin {proof} These assertion are equivalent via the $*$ operation.  The last part of (i) follows from the first part and the left boundary condition. Suppose $s:=ht C_1<r$.  By the left boundary conditions  $s$ must be even.  Let $i$ be minimal such that $s':=ht C_i >s$. Then $s'=s+1$, otherwise we could adjoin a left even domino to $C_1$. This forces $s'$ to be odd.  Then by the left boundary condition $s'=r$.
\end {proof}

\subsubsection{The Height Function of a Complete Diagram }\label{2.3.3}

\begin {lemma}  Let $\mathscr D$ be a complete diagram.  Then for all $i\in T$

\

(i)  $|ht_\mathscr D (i+1)-ht_{\mathscr D}(i)|\leq 2$.

\

(ii) If $ht_\mathscr D (i+1)$ is even, then $ht_\mathscr D(i+1) -ht_{\mathscr D}(i)\leq 1$.

\

(iii) If $ht_\mathscr D (i-1)$ is odd, then $ht_\mathscr D(i-\mathbf{}1) -ht_{\mathscr D}(i)\leq 1$.
\end {lemma}

\begin {proof}  Fix $i =1,2,\ldots,t$ and set $s = ht C_i$. Assume that $ht C_{i+1} >s$.
%Assume that $ht C_{i+1} >s$. the height function strictly increases to the right of $C$, that is to say the column $D_r:=C_{i+1}$ has height $\geq s+1$ and so in particular is the right neighbour of $C$ at level $s$.
% with $s$ a local minimum of the height function.

Assume that $s$ is even.  By completeness $ht C_{i+1}\leq s+1$. This proves (ii). Through $*$, we obtain (iii) from (ii).

Assume $s$ is odd. We claim that $C_{i+1}$ to $C$ has height at most $s+2$.  If not, $C_{i+1}$ has height $>s+2$ and then by Lemma \ref {2.3.2} there is left neighbour $C^\ell$ of $C_i$ at level $s$ of height $s'\geq s+1$. If $s'>s+1$, then one may add a vertical domino to $C_i$ contradicting completeness. Suppose $C^\ell$ has height $s'=s+1$.  Since $C_{i+1}$ has height $\geq s+3$, then one may add a vertical domino to $C^\ell$ again contradicting completeness. Combined with $*$  and (ii), (iii), this proves (i).
%Now assume that $C_{i-1}$ has height $\geq s+1$ so in particular is the left neighbour of $C_i$ at level $s$. By $*$  and the first result we conclude that if $s$ is odd (resp. even) then $C^\ell$ has height at most $s+1$ (resp. $s+2$).

%Thus if the height function takes an even value it may increase by at most one (resp. two) on passing to the right (resp. left) and if it takes an odd value it may increase by at most two (resp. one) on passing to the right (resp. left).
\end {proof}

\textbf{Remark}.  This result can be expressed by saying that if the height function takes an even value it may increase by at most one (resp. two) on passing to the right (resp. left) and if it takes an odd value it may increase by at most two (resp. one) on passing to the right (resp. left).  In particular it cannot increase by two or more on passing to both the left and the right of a given column.

\subsubsection{Half-domino adjunction}\label{2.3.4}

\

\

\textbf{Definition.} A column $C$ of a non-empty diagram $\mathscr D$ is said to be strongly extremal if it is of the maximal height $r:=ht \mathscr D$  and  is left (resp. right) extremal for $r$ odd (resp. even).

\

 This can also be expressed as saying that for $r$ odd (resp. even) $C$ is the leftmost (resp. rightmost) column of $\mathscr D$ of height $r$. As a consequence $\mathscr D$ admits exactly one strongly extremal column.  For the empty diagram we designate $C_{t+1}$ as its unique strongly extremal column.  This is compatible with its equivalence to the diagram with an even number of rows all having $t+1$ blocks.

One readily checks that if $C$ is strongly extremal, then it remains so after adjoining any left even or right odd domino to a column of $\mathscr D$ of height $\leq ht \mathscr D -2$.  Thus all diagrams in the weak equivalence class of $\mathscr D$  have the same strongly extremal column.

A single block may be adjoined to $C$ of maximal height to obtain a column $C^+$ in a diagram $\mathscr D^+$ satisfying the boundary conditions if and only if $C$ is strongly extremal. (Otherwise for $r$ odd (resp. even) a column which is a left (resp. right) neighbour to $C$ at level $s$ will violate the left (resp. right) boundary condition in $\mathscr D^+$.)  We call this operation half-domino adjunction to distinguish it from ``single block linkage" which is defined in  \ref {3.3}.

On the other hand $\mathscr D^+$ need not be complete.  By the previous paragraph its completion $\widehat{\mathscr D^+}$ can be obtained by first taking the completion $\hat{\mathscr D}$ of $\mathscr D$ adding an extra block to $C$ to obtain $\hat{\mathscr D}^+$ (since the latter can be obtained by adjoining suitable dominos to $\mathscr D^+$) and completing.

\begin {lemma}  Suppose $\mathscr D$ is complete.  If $r:=ht \mathscr D$ is even (resp. odd) then $\widehat{\mathscr D^+}$ is obtained from $\mathscr D^+$ by adjoining a domino to every column strictly to the right (resp. left) of $C^+$.
\end {lemma}

\begin {proof}  Through the $*$ operation we can assume that $r$ is even.  Since $\mathscr D$ is complete a domino $D=D_i^{i+1}$ can be adjoined to $\mathscr D^+$ only if $i=r$. Then $D$ must be an odd domino, hence an odd right domino and therefore can only be added to columns lying strictly to the right of $C^+$.

Since $C$ is strongly extremal the columns to the right of $C^+$ form a diagram $\mathscr D^+_1$ with $C^+$ of height $r+1$ and with every other column of height $\leq r-1$. Then by Lemma \ref {2.3.2}, since $\mathscr D$ is complete, we obtain $ht C_{t+1}=r-1$. Since $r-1$ is odd and $C^+$ has height $r+1$, we may adjoin a right odd domino to $C_{t+1}$ and indeed to every column of height $r-1$ strictly to the right of $C^+$ to obtain a new diagram $\mathscr D^+_2$ in which every column has height $r+1$, or height $\leq r-2$ and with both left and right neighbours of height $r+1$.  Since $r-2$ is even, we may adjoin a left even domino to every column of height $r-2$ to obtain a diagram $\mathscr D^+_3$ in which every column has height $r$, or height $\leq r-3$ and with both left and right neighbours of height $r$.  Continuing in this fashion we eventually obtain a diagram $\mathscr D^+_\ell$ in which to every column of $\mathscr D^+_1$ there has been adjoined a vertical domino.  Then $\mathscr D^+_\ell$ differs from $\mathscr D^+_1$ just by adding two complete rows at the foot of the columns lying strictly to the left of $C^+$.  Then $\mathscr D$ being complete implies that $\mathscr D^+_\ell$ is complete.
 %and that adjoining the columns strictly to the left of $C^+$ gives the completion of $\mathscr D^+$.
\end {proof}

\subsubsection{Repeated half-domino adjunction}\label{2.3.5}

Retain the above notation.  We may repeat the process of half-domino adjunction to $\widehat{D^+}$.  In this $C^+$ is again a strongly extremal column in $\widehat{\mathscr D^+}$ which is now of height $r+1$ which is odd (resp. even).  Then through the lemma we obtain a diagram in which \textit{every} column of $\mathscr D$ has been adjoined a vertical domino.  It can hence be obtained by just adding two complete rows to the bottom of the diagram.  The same diagram is also obtained by adjoining a vertical domino to the strongly extremal column and completing.

\subsubsection{Equal Adjacent Rows}\label{2.3.6}

If two adjacent rows $R_{i+1},R_{i+2}$ admit the same number of boxes, then the dominoes $D^{i+2}_{i+1}$ in each column may be cancelled collapsing at the same time the diagram to one of height two less without upsetting the boundary conditions.   A diagram in which no two adjacent rows can be cancelled is called reduced.

\begin {lemma}

\

(i) A deplete diagram is reduced.

\

(ii) A complete diagram is reduced if its first two rows do not coincide..
\end {lemma}

\begin {proof}  (i). Indeed suppose we have two consecutive rows $R_i,R_{i+1}$ in a depleted diagram $\mathscr D$ having the same cardinality.  This means that there is no column of height $i$ in $\mathscr D$. We can assume that the length of a row $R_{i+2}$ is strictly smaller.  This means that there is a column $C$ of $\mathscr D$ of height $i+1$.  If $C$ is not an extremal column, then a vertical domino can be removed from its top.  On the other hand if $C$ is a column to the extreme left (right) of $\mathscr D$, then since it is not the highest column it must have even (resp. odd) height and again a left and even (resp. right and odd) vertical domino can be removed from $C$.  This contradicts $\mathscr D$ being deplete.  Hence (i).

(ii).  Otherwise the height function would at some point increase by at least two on passing to both left and right. For a complete tableau this is excluded by Lemma \ref {2.3.2}  (see Remark).

\end {proof}

\subsubsection{Equivalence Classes}\label{2.3.7}

The above three operations: adjunction (and removal) of dominos (and half-dominos) and adjunction (and removal) of equal adjacent rows may be combined to give an equivalence relation on diagrams.  by Lemmas \ref {2.3.1}, \ref {2.3.6} there is exactly one deplete diagram in each equivalence class.  By contrast

\begin {lemma} Every equivalence class of diagrams admits exactly two reduced complete diagrams.  Their heights differ by $1$.  In an equivalence class of diagrams there is at most one complete diagram of height $s \in \mathbb N$ and exactly one if $s$ is sufficiently large.
\end {lemma}

 \subsubsection{Duality}\label{2.3.8}

 After elimination of adjacent rows the map $\mathscr D \rightarrow \mathscr D^*$ becomes involutive.  We call it duality.  Duality increases or decreases height by one.  However due to equivalence there can be self-dual diagrams.

 We may represent a diagram by the unordered partition it defines.  However equivalence classes and duality are more easily pictured through the diagram itself.  The following are simple examples. For $t=2$ consider the deplete diagram $(0,1,1)$.  Its dual is $(2,2,1)$.  The strongly extremal column of $(0,1,1)$ is $C_2$.  Adjunction of a half domino gives $(0,2,1)$ which becomes $(2,2,1)$.  On the other hand for $t=3$ the deplete diagram $(0,1,1,1)$ has dual $(0,0,1,1)$ as its deplete representative.

 \subsection{}\label{2.4}

 The set of equivalence classes of diagrams of order $t+1$ is denoted by $H^{t+1}$, or simply by $H$.  For any representative $\mathscr D \in H$ the $x$ co-ordinate of a strongly extremal column (see \ref {2.3.4}) is independent of the choice of representative.  Then for any $j \in \hat{T}$ we let $H_j$ denote the subset of $\mathscr D \in H$ for which the strongly extremal column has $x$ co-ordinate $j$.

 \section{Tableaux}\label{3}

 \subsection{Numbering}\label{3.1}

 A tableau $\mathscr T$ is a diagram (satisfying the boundary conditions) in which the blocks have  entries given by the following rules.  Here it is always assumed that a block belongs to some column.  More precisely a block is said exist at the $(i,j)$ co-ordinate when  $ht C_i \geq j$. In this case it will be denoted by $B(i,j)$. The entry in $B(i,j)$ is set equal to $b(i,j)$.

 In particular the leftmost (resp. rightmost) block at level $j$ is $B(i,j)$, where $i$ is minimal (resp. maximal) such that $ht C_i \geq j$.  It is called \textit{the} extremal block at level $j$ if $j$ is even (resp. odd).  In particular for a non-empty diagram, $B(t+1,1)$ is the extremal block at level $1$.   Set $I_0=T$ and let $I_j$ be the subset $\hat{T}$ such that $B(i,j)$ is a block and not extremal.

 \

\textbf{ N.B.}  The extremal blocks lie in the extremal columns.  However they do not lie in the unique strongly extremal column unless $j = ht \mathscr T$ \textit{and} there is just one block in $R_j$.

 \

 (i)  A diagonal slash is inserted into \textit{the} extremal block of every non-empty row.

 \

 (ii) $i$ is inserted into $B(i,1):i \leq t$, that is $b(i,1)=i: i \in T$.

 \

 (iii) Insertions in subsequent rows are defined inductively as follows.  Suppose $j$ is odd (resp. even) and that insertions in $R_j$ have been determined.  Assume $ht C_i \geq j+1$, let $C_k$ be the left (resp. right) neighbour to $C_i$ at level $j$.   Assume that $B(i,j+1)$ is not an extremal block.  One checks from the boundary conditions that $B(k,j)$ is a block and not extremal.  Then we set $b(i,j+1)=b(k,j)$.

 \

 More formally  for each $j \in \mathbb N^+$ we have an injective map $\varphi_j:I_{j+1}\rightarrow I_j$ defined by setting $\varphi_0=\Id$ and $b(i,j+1)=b(\varphi_j(i),j)$ for $j>0$.  Set $\theta_j=\varphi_0\varphi_1\ldots \varphi_j$. This gives an injective map
$\theta_j:I_{j+1}\rightarrow T$ such that $b(i,j+1)=\theta_j(i)$.

 \

 It is immediate from the above construction that entries strictly increase in rows. This is why we do not insert $0$ into the extremal blocks, though we do this in \ref {5.4}.

 \

Since these rules uniquely determine the entries in the blocks of a diagram we may identify diagrams and tableaux.  Thus for example a tableau of a complete diagram will be called a complete tableau.

\

 These rules applied to the dual diagram and extends duality to tableaux.

 \

 A tableau $\mathscr T$ is said to be well-numbered if, for a non-extremal block, one has

 $$b(i,j)= \left\{\begin{array}{ll}i& : j \ \text{ odd}, \\
   i-1&:j \ \text{even}.\\

\end{array}\right.$$

We extend this rule for $j=0$.  Note that this rule is compatible with adding two complete rows at the bottom of a diagram.
 %$B_{i,j}$ equals $i$ for $j$ odd and $i-1$ for $j$ even.

 From the height function of a complete diagram $\mathscr D$ one easily checks (see Remark \ref {2.3.3}) the

 \begin {lemma} A complete tableau is well-numbered.
 \end {lemma}

 \subsection{Ordering}\label{3.2}

 \subsubsection{Partial Order}\label{3.2.1}

 To a tableau $\mathscr T$ we may assign a partial order $P(\mathscr T)$ as follows.

 Assume that $j$ is even.  With $s$ determined as below, fix a column $C_{i_s}$ with a left neighbour $C_{i_1}$ at level $j$.
 %Let $c_{i_s}^j$ be the entry in the box of $C_{i_s}$ at level $j$.
 By our construction  $b(i_s,j)=b(i_{s-1},j-1)$ where $C_{i_{s-1}}$ is the left neighbour of $C_{i_s}$ at level ${j-1}$.  One has $i_1\leq i_{s-1} <i_s$.  Let $C_{i_2}, C_{i_3}, \ldots, C_{i_{s-1}}$ be the columns of height $j-1$ lying between $C_{i_1},C_{i_{s}}$ (thus determining $s$) and set %$c_{i_1}^{j-1},c_{i_2}^{j-1},\ldots,c_{i_{s-1}}^{j-1}$ be their entries in the $(j-1)^{th}$ level.  Then  Lemma \ref {6.4.24}$(i)$ translates to give the inequalities
$$b(i_v,j-1) > b(i_s,j), \forall v=1,2,\ldots,s-2. \eqno {(*)}$$

\

 Assume $j$ odd. Fix a column $C_{i_1}$ with a right neighbour $C_{i_s}$ at level $j$.  By our construction  $b(i_1,j)=b(i_2,j-1)$ with  $C_{i_{2}}$ right neighbour of $C_{i_1}$ at level ${j-1}$.  One has $i_1 <i_{2} \leq i_s$.  Let $C_{i_2}, C_{i_3}, \ldots, C_{i_{s-1}}$ be the columns of height $j-1$ lying between  $C_{i_1},C_{i_{s}}$ (thus determining $s$) and set

$$b(i_v,j-1) > b(i_1,j), \forall v=3,\ldots,s. \eqno {(**)}$$

We define the relations in  $P(\mathscr T)$ to be the above with $j \in \{1,2,\ldots, ht \mathscr T\}$

Since a diagram gives rise to exactly one tableau, the above relations can also be associated to the corresponding diagram.

If two consecutive rows are adjoined to a diagram, then the relations above are obviously unchanged.  However relations \textit{are} changed by adjunction of vertical dominoes.  On the other hand since the behaviour of the height function is very similar for all completed diagrams in an equivalence class, the relations for completed diagrams in an equivalence class coincide and as we shall see take a particularly simple form.   We shall eventually show (Lemma \ref {5.3}) that for an arbitrary diagram in an equivalence class the relations include all the relations for any complete diagram in that class.

\

These relations are compatible with duality.  More precisely the relations coming from $\mathscr T^*$ are obtain from those obtained from $\mathscr T$ by the substitution $i \mapsto t+1-i: i \in T$.

\subsubsection{Linear Order}\label{3.2.2}

It is rather easy to show that the graph of $P(\mathscr T)$ has no cycles and so can be lifted to a linear order $L(\mathscr T)$ and indeed to several.

Fix a tableau $\mathscr T$ and set $T^0= \hat{T}$ and $T^j:=\{b_{i,j}\}_{i \in \hat{T}}: j \in \mathbb N^+$. By the construction of \ref {3.1}, the $T^j$ are decreasing in $j$.   For all $j'<j \in \mathbb N$, set $\overline{T}^{j'}=T^{j'}\setminus T^j$.

%Now let for $j$ even (resp. $j$ odd) let $N^{j-1}$ denote the subset of $\overline{T}_{j-1}$ given in the left hand side of \ref {3.2.1}$(*)$ (resp. $(**)$).

The relations given in \ref {3.2.1}$(*),(**)$) are \textit{amongst} those of the form $b<b': b\in T^j, b' \in \overline {T}^{j-1}$.  For the graph of $P(\mathscr T)$ to admit a cycle we would need $T^j\cap \overline{T}^{j'}\neq \phi$ for some $j'<j$. We conclude that there are no cycles in $P(\mathscr T)$ and so it can be lifted to a linear order.

For a complete tableau we can do much better.  We say that a relation of the above form is superfluous if $j \in \mathbb N^+$ is not maximal with the property that $b \in T^j$.
%Since the $\overline{T}^{j-1}$ are pairwise disjoint an element of $T$ can occur only once as the smaller (resp. larger) element and in particular only twice in the set of inequalities.

\begin {lemma}  For a well-numbered tableau $\mathscr T$ there are no superfluous relations.
\end {lemma}

\begin {proof}  Fix $i$ appearing in $\mathscr T$ and let $j$ be maximal such that $i$ appears in the row $j$ of $\mathscr T$.  Then $i$ appears in row $j_-$  of $\mathscr T$, for all $j_- \leq s$.  We must show that only from the row $j$ can we obtain an inequality of the form $i < k$ by the procedure of \ref {3.2.1}.

Suppose $j$ is even.

Since $\mathscr T$ is well-numbered, $i$ appears in $C_{i+1}$ at row $j$ and in $C_i$ at row $j-1$.  In particular $C_i$ has height $\geq j-1$.

Suppose $j_-$ is even.  Since $\mathscr T$ is well-numbered, $i$ appears in $C_{i+1}$ at row $j_-$ and in $C_i$ at row $(j_--1)$.  Yet $C_i$ has height $j-1>j_-$, so \ref {3.2.1}$(*)$ yields no inequalities with $i=i_s, j=j_-$.

Suppose that $j_-$ is odd. Since $\mathscr T$ is well-numbered, $i$ appears in $C_i$ at row $j_-$ and in $C_{i+1}$ at row $(j_--1)$. Yet $C_{i+1}$ has height $\geq j>j_-$, so \ref {3.2.1}$(**)$ yields no inequalities with $i_1=i,j=j_-$.

The case when $j$ is odd obtains from duality.

\end {proof}

 \textbf{Remark}.  Suppose that $\mathscr T$ is well-numbered, for example complete. Then since the $\overline{T}^{j-1}$ are pairwise disjoint an element of $T$ can occur only once as the smaller (resp. larger) element and in particular only twice in the set of inequalities.

 \subsubsection{Example.}\label{3.2.3}

 The deplete diagram $(2,1,0,2)$ gives gives to the relations $2<1,2<3$ which exhibits a superfluous inequality.  It completion is $(2,1,2,2)$ which rise to just $2<1$.  It is very rare that a complete tableau gives a linear order. Indeed the above result implies that this can only happen if it has height $\geq t-1$, whilst a complete diagram of minimal height in its equivalence class must have height $\leq t$.  If  $t$ is even, then the diagram $(t,t-2,\ldots,2,1,3,\ldots,t-1,t)$ yields the linear order $t < 1 < t-1 < 2 < t-2 < \ldots < t/2$, whilst if $t$ is odd the diagram $(t,t-1,t-3,\ldots,2,1,3,\ldots,t)$ yields the linear order $1 < t < 2 < t-1< \ldots < (t+1)/2$.

  \subsection{Single Block Linkage}\label{3.3}

  Fix $j,k \in \hat{T}$ distinct.  Single block linkage is a map $H_j \rightarrow H_k$ obtained by adjoining a single block to the top of the column $C_k$ in a complete reduced diagram $\mathscr D \in H_j$ to give a diagram $\mathscr D \in H_k$.  This construction is made precise below.

  Recall that we identify diagrams and tableaux.

  \subsubsection{Quasi-extremal Columns}\label{3.3.1}
  
  Quasi-extremal columns are only defined for complete tableaux.

  Let $\mathscr T \in H_j$ be a complete  tableau (or diagram)  of height $s$.  A left (resp. right) quasi-extremal column $C$ of $\mathscr T$ is a column which is not strongly extremal such that every column to the left (resp. right) of $C$ has height at most that of $C$.  By Lemma \ref {2.3.2} the height of $C$ must be at least $s-1$.

Suppose that $s$ is odd (resp. even). Then $C_{t+1}$ (resp. $C_1$) has height $s$. Moreover $C_j$ has height $s$ and by definition (\ref {2.3.4}) is the leftmost (resp. rightmost) column of $\mathscr T$ of height $s$.

Suppose $C$ has height $s-1$.  Then it can only be a leftmost (rightmost) quasi-extremal column of $T$ and must lie to the left (resp. right) of $C_j$.

Suppose $C$ has height $s$. Then it must lie to the right (resp. left) of $C_j$.

Finally one notes that $C_j$ is a quasi-extremal column of $\mathscr T$, if and only if $C_{t+2-j}$ is a quasi-extremal column of $\mathscr T^*$.

  \subsubsection{}\label{3.3.2}

 \begin {lemma} Let $\mathscr T'\in H_j$ be a complete reduced tableau of height $s'$. A tableau $\mathscr T \in H_k:k \neq j$ may be obtained by adding a single block to the top of some column $C'_k$ of $\mathscr T'$ if and only if $C'_k$ is quasi-extremal of height $s'-1$.
 \end {lemma}

 \begin {proof}

 Let $C_k$ be the column in $\mathscr T$ obtained by adding a block to $C'_k$.  Since $C'_k$ is not strongly extremal, whilst $C_k$ is required to be strongly extremal, it follows that $ht C'_k=s'-1$ by \ref {2.3.4}.

 Suppose that $s'$ is odd (resp. even) in what follows below.

  Then by definition (\ref {2.3.4}) it follows that $C_k$ has no neighbour to the left (resp. right) at level $s'$. It follows by definition (\ref {3.3.1}) that $C'_k$ is left (resp. right) quasi-extremal.

  Conversely if $C'_k$ is quasi-extremal of height $s'-1$ it follows by definitions (\ref {3.3.1},\ref {2.3.4}) that $C_k$ is strongly extremal.

  \end {proof}

  \textbf{Remark}.  It is also possible to add a block to a quasi-extremal column column $C_k'$ of height $s'$ if we \textit{add at the same time} a block to the strongly extremal column $C'_j$.  However the resulting tableau when completed is the same as if we add a block to $C'_j$ complete the tableau, proceed as in the lemma above and finally complete the resulting tableau.  In other words because there are two complete reduced tableaux in each equivalence class, this second possibility does not have to be considered.

  \subsubsection{The Graph of Links}\label{3.3.3}

  Recall the construction of \ref {3.3.2}.  The block added to $C_k'$ to obtain $\mathscr T$ cannot be extremal and so contains an element $i \in T$ determined by the rules in \ref {3.1}.

  Thus single block linkage gives a graph $\mathscr G_H$.  We have a map $v\mapsto i_v$ from the set $V(\mathscr G_H)$ of vertices of $\mathscr G_H$ to $\hat{T}$ and a map $(v,v') \mapsto i_{(v,v')}$ from the set of edges $E(\mathscr G_H)$ of $\mathscr G_H$ to $T$ given as follows.

   If $\mathscr T' \in H_j$, then the corresponding vertex $v'$ is labelled by $j$, that is $i_{v'}=j$.

     Given a single block linkage from $\mathscr T' \in H_j$ to $\mathscr T \in H_k$ the edge $(v,v')$ joining $v'$ to $v$ is labelled by the entry $i$ of the added block that is $i_{(v,v')}=i$.

     Observe that $i_{(v,v')}$ is determined  by the pair $\mathscr T,\mathscr T'$.  Consequently there can be at most one edge with a given label emanating from a given vertex. Again $j,k$ are distinct, so that the labels on vertices joined by an edge must be distinct.  Further rules will be described in Section \ref {6}.

  \section{Functions}\label{4}

\subsection{}\label{4.1}

In this section we use the formalism but not the results of the Sections \ref {2}, \ref {3} to define for each tableau $\mathscr T$ a function $f_\mathscr T$ which is separately linear in two sets of variables respectively labelled by $T,\hat{T}$.  We shall show that $f_\mathscr T$ is independent of the choice of $\mathscr T$ in its equivalence class and that $f_\mathscr T \neq f_{\mathscr T'}$ for distinct deplete tableaux.

\subsection{Further Notation and Motivation}\label{4.2}

Let $\textbf{c}:=\{c_i\}_{i\in T}$ be indeterminates eventually replaced by \textit{non-negative} integers coming from entries in the Cartan matrix.  Then inequalities between the resulting coefficients taken in general position (that is when $c_i\neq c_j,$, for all $i \neq j$) define a linear order $L(T)$  on $T$.  Conversely a linear order $L(T)$ on $T$ defines a \textbf{sector} $\mathscr S_{L(T)}:=\{\textbf{c}|i<j \Rightarrow c_i < c_j\}$.  This decompose $\mathbb N^t$ into $t!$ sectors which meet at boundaries.  To economise on notation we use $\textbf{c}$ to denote the sector defined by a linear order on $T$.

Though this will not concern us in the present paper we remark that for a fixed simple root $\alpha$, the $-c_i:i \in T$ are given by sums of entries  of the Cartan matrix coefficients $\alpha^\vee(\beta): \beta \in \pi$ with non-negative integer coefficients.

Let $\{m^i\}_{i \in \hat{T}}$ be indeterminates.   They represent $t+1$ consecutive places in $B_J$ corresponding to $\alpha$. Let $\{r^i\}_{i \in \hat{T}}$ be the Kashiwara functions corresponding to $\alpha$ and these places.  For a given $b \in B_J$ they become non-negative integers.  One may write $r^i-r^{i+1}=m^i+m^{i+1}+M_i$ for all $i \in T$.  Here the $M_i$ are integers obtained from the entries of $B_J$ corresponding to the remaining simple roots $\beta \in \pi \setminus \{\alpha\}$.

The signs of the differences  $r^i-r^{i+1}: i \in T$ determine the manner in which the Kashiwara operators $e_\alpha,f_\alpha$ act on $B_J$.  Thus $M_i$ being possibly negative balances off the positivity of the sum $m^i+m^{i+1}$ and gives the structure of $B_J(\infty)$ its extraordinary subtlety.  However this  will only be implicit in the present work and so here $M_i$ is set equal to zero.

The driving function $h$ is given by the previous induction step (involving a different simple root) and takes the form $h=-\sum_{i\in T}c_im^i$.  It will be represented by the empty tableau $\mathscr T_h$.   The unique strongly extremal column (see \ref {2.3.4}) of $\mathscr T_h$ is $C_{t+1}$.  This corresponds to the fact that the coefficient of $m^{t+1}$ in $h$ is zero.
Thus $h$ does not depend on $m^{t+1}$ and so will not change if $e_\alpha$ (resp. $f_\alpha$) alters the value of $m^{t+1}$, that is to say ``enters at the $(t+1)^{th}$ place".
%which holds $r^{t+1}\geq \Max _{i \in \hat{T}}r^i$ (resp. $r^{t+1}> \Max_{i \in \hat{T}}r^i$).

The dual Kashiwara functions taken with respect to the fixed $\alpha$ need to exhibit a similar invariance property but must allow the entry of the Kashiwara operators at any place.

 Let $\hat{H}^{t+1}$ (or simply, $\hat{H}$) be the $\mathbb Z$ module set of ``bilinear'' functions $\sum_{i\in \hat{T}}d_im^i$ with $d_i \in \sum_{j \in T}\mathbb Zc_j$. To obtain the required ``invariance'' we construct a subset $H \subset \hat{H}$ with following properties.

\

 Set $H_j:=\{f \in H|d_j=0\}$. Then

 \

 (i)  $H=\cup_{j \in \hat{T}} H_j$.

 \

(Eventually $H,H_j$ will be identified with the objects defined in \ref {2.4}.)

\

  For each sector $\textbf{c}$ we define in \ref {5.3} a subset $H(\textbf {c})$.  It has the property that $H(\textbf {c})\subset\{f \in H| f = h+\sum_{i \in T}\mathbb N(r^i-r^{i+1})\}$. We say that $f\in H$ satisfies the \textit{positivity condition} with respect to a sector $\textbf{c}$ if $f \in H(\textbf{c})$.

  \

 (ii) $H=\cup H(\textbf{c})$, the union being over all sectors.

 \

  Set $H_j(\textbf{c})=H_j\cap H(\textbf{c})$.

\

(iii)   $\Max_{f \in H(\textbf{c})}f$ is unchanged when $e_\alpha$ (resp. $f_\alpha$) enters the $j^{th}$ place.

\

%To ensure this property we need in particular to be an $f\in H_j(\textbf{c})$ which is maximal in $H(\textbf{c})$ under the above conditions on the Kashiwara functions.

%\

 Here we shall only need to know that (iii) can be ensured by a condition, which we call condition $S$ on $H(\textbf{c})$.  We construct $H(\textbf{c})$ with this in mind using the formalism of the previous sections and then the verification of condition $S$ will constitute the proof of the Preparation Theorem (Theorem \ref {8.6}).  Nevertheless in \ref {6.7} we shall discuss briefly how condition $S$ ensures this result.

 \

 Though this will not concern the present work, let us briefly discuss the role of tableaux in the Preparation Theorem.  The unique tableau of height $0$ corresponds to the driving function $h$. The tableaux of height $1$, which correspond to adding successive differences of Kashiwara functions to $h$, are required for invariance under $e_\alpha$ action.  Subsequently the tableaux of height $2$, which correspond to subtracting successive differences of Kashiwara functions from the functions obtained from the tableaux of height $1$ are required for invariance under $f_\alpha$ action.  The positivity condition implies that the overall sum is still a sum with positive coefficients of successive differences of Kashiwara functions, but now this can only be true for certain sectors.  A further sweep of $e_\alpha$ insertion gives the tableaux of height $3$ and so on. (I did not originally anticipate the need for diagrams of height $>2$, but an example of P. Lamprou showed these to be unnecessary.  The first example is $(3,2,1,3)$.)

 The essence of the Preparation Theorem is that (iii) holds.  In this it is a remarkable fact that we need to impose positivity conditions which are stronger than those which the above condition might seem to imply.  They result from a duality condition whose necessity is also not immediately obvious.  All that we can say is that this procedure works! In terms of the $S$-graphs introduced in Section \ref {7} these extra conditions are automatically included. However it is only in terms of tableaux that one sees (see Section \ref {5.5}) that when duality is incorporated these positivity properties are the minimal possible in each sector. For a simple Lie algebra of type $A$ one need never go beyond tableaux of height $1$.  However tableaux of height $>1$ are needed, probably for the first time, in type $D_5$.
 
 Another importance of the tableaux is that $S$-graphs are not uniquely determined; but the $S$-graphs are canonically determined as subgraphs of the graph of links between the tableaux \cite {JL2} and it is these canonically determined $S$-graphs which are to be used in constructing the dual Kashiwara functions.

 %Because the linear order on $\textbf{c}$ and the desired order on $ H_j(\textbf{c})$ have to be balanced off, it can anticipated that the verification of condition $S$ will not be easy.  Indeed the most difficult part of the present work was to guess a suitable choice of $H$.  We shall not encumber the present paper with any clues as to how this was achieved.

\subsection{Duality}\label{4.3}

Set $c_i^*=c_{t+1-i}: i \in T, (m^i)^*=-m^{t+2-i}:i \in \hat{T}$.  Extend $*$ to products as an automorphism.  Extend $*$ by $\mathbb Z$ linearity to the $\mathbb Z$ module $\hat{H}$.  It is clear that $*$ is an involution of $\hat{H}$.  It is called duality.  Since $r^i-r^{i+1}=m^i+m^{i+1}$, for all $i \in T$ we may further set $(r^i)^*=r^{t+2-i}:i \in \hat{T}$.  Recall that $h=-\sum_{i=1}^tc_im^i$. One checks that $h^*=h+\sum_{i=1}^tc_i(r^i-r^{i+1})$.

\subsection{Reading the Tableau}\label{4.4}

Recall \ref {3.1} and let $\mathscr T$ be a tableau.  Fix a row $R_s$ of $\mathscr T$ and set $|R_s|=k_s$.  Let $\{C_{u_i}\}_{i=1}^{k_s}$ be the set of columns of $\mathscr T$ of height $\geq s$.  Suppose $s$ is odd (resp. even) and recall that every entry except the last (resp. first) row has an entry $b(u_i,s) \in T$. Set
$$f_{R_s}= \sum_{i=1}^{k_s-1}c_{b(u_i,s)}(r^{u_i}-r^{u_{i+1}}), \quad (\text{resp}. \  f_{R_s}= \sum_{i=2}^{k_s}c_{b(u_i,s)}(r^{u_i}-r^{u_{i-1}})).\eqno{(*)}$$

Then
$$f_\mathscr T:= h +\sum_{s=1}^{ht \mathscr T}f_{R_s}. \eqno {(**)}$$

%Set $c_i^*=c_{t+1-i}: i \in T, (r^i)^*=r^{t+2-i}:i \in \hat{T}$.  Extend $*$ to products as an automorphism.  Finally Set $h^*=h+\sum_{i=1}^tc_i(r^i-r^{i+1})$.  Extend $*$ by linearity to all sums of the form $h+ \sum_{i \in T}\sum_{j,k \in \hat{T}}c_i(r^j-r^k)$.

\begin {lemma} For every tableau $\mathscr T$ one has $f_{\mathscr T^*}=f^*_\mathscr T$.
\end {lemma}

\begin {proof} This follows from the definition of $f_\mathscr T$, the $*$ operation on diagrams given in \ref {2.2} and the assignment of entries in tableaux given in \ref {3.1}.
\end {proof}

\subsection{Class Function Property}\label{4.5}

\begin {lemma}  $f_\mathscr T$ is independent of the choice of $\mathscr T$ in its equivalence class.
\end {lemma}

\begin {proof} Obviously $f_{R_s}=0$ if $k_s=1$.  Suppose $k_s=k_{s+1}$. This means that $B(u_j,s+1)$ is a block if and only if $B(u_j,s)$ is a block. Then by \ref {3.1}(iii), we obtain $b(u_{i+1},s+1)=b(u_i,s)$ (resp. $b(u_{i-1},s+1)$), for $s$ odd (resp. even). Consequently through $(*)$ we obtain $f_{R_{s+1}}+f_{R_s}=0$.

It remains to consider domino adjunction, which is slightly more subtle.

Consider an odd right domino $D$ being removed from column $C_b$ of height $s$ in $\mathscr T$.   By definition $s$ is odd and $C_b$ admits in levels $s,s-1$ a left neighbour $C_a$.  In particular $ht C_a \geq s$.  By \ref {3.1}(iii) one has $b(s,a)=b(s-1,b)$, since $s$ is odd.  If $C_b$ has a right neighbour $C_c$ at level $s-1$, then since $s-1$ is even and strictly less than the height of $\mathscr T$, the right boundary condition on $\mathscr T$ implies that $C_b$ has a right neighbour $C_d$ at level $s$.  The converse is immediate.  Obviously $c\leq d$ and from  \ref {3.1}(iii) we obtain $b(s,b)=b(s-1,c)$.

When $C_b$ has no right neighbour at level $s-1$ and so a fortiori no right neighbour at level $s$, the assertion is a easy case of what follows below and its proof will be omitted.

The contribution coming from $D$ as given by \ref {4.3}$(*),(**)$ is
$$b(s-1,c)(r^c-r^b)+b(s-1,b)(r^b-r^a)+b(s,a)(r^a-r^b)+b(s,b)(r^b-r^d)=b(s-1,c)(r^c-r^d).$$

Now in the removal of $D$ from $\mathscr T \setminus D$ we must not forget! that the \ref {3.1}(iii) implies a change of the entry $b(s,a)$ which becomes $b(s-1,c)$.  Then removal of $D$  gives a new term
$$b(s-1,c)(r^c-r^a)+b(s,a)(r^a-r^d)=b(s-1,c)(r^c-r^d).$$

The coincidence of these two terms proves our assertion for the removal (and hence also for adjunction) of odd right dominoes.  Then duality gives the assertion for even left dominoes.

\end {proof}

\subsection{Support}\label{4.6}

View the $c_i:i \in T$ as indeterminates. Given a polynomial $f$ in the $\{c_i\}_{i\in T}$, define $\Supp f$ as the set of all $i\in T$ such that $f$ depends non-trivially on $c_i$.

 Define $\Supp \mathscr T$ as the set of $i\in T$ such that the column $C_i$ is non-empty.   It is immediate from \ref {3.1} that $\Supp (f_\mathscr T -h)\subset \Supp \mathscr T$.  However equality generally fails. For example take $\mathscr T$ to be defined by the partition $(2,1,2,2)$ which is a complete tableau.

 \begin {lemma} Let $\mathscr T$ be a deplete tableau.  Then  $\Supp (f_\mathscr T-h) = \Supp \mathscr T$.
 \end {lemma}

 \begin {proof}

  To prove our claim suppose that $c_i \notin \Supp (f_\mathscr T-h)$, for some $i \in T$. We must show that $C_i$ is empty.
Suppose $C_i$ has height $s>0$.  By \ref {3.1}(ii), the entry in the first row of $C_i$ equals $i$.

  By the right boundary condition there exists a right neighbour $C_{i_1}$ to $C_i$ at level $1$.  Then the first rows of these two columns contribute a factor of $c_i(r^i-r^{i_1})$ to $f_\mathscr D$. By our hypothesis  $c_i \notin \Supp (f_\mathscr T-h)$ and so both $c_ir^i$ and $c_ir^{i_1}$ must cancel with the remaining terms $(f_\mathscr T-h)$.  Since entries strictly increase in rows, it follows that these terms must come from levels $>1$. Thus $C_i,C_{i_1}$ must both have height $\geq 2$ and then by \ref {3.1}(iii) that the entry in the second row of $C_{i_1}$ is $i$.  Then indeed $(f_\mathscr T-h)$ admits the factor $c_i(r^{i_1}-r^{i})$ coming from the second level and which cancels with the factor given above.

 %If $s>2$, then by \ref {3.1}(iii), $i$ is the entry of the third row of $C_i$.   By the right boundary condition there is a right neighbour $C_{i_2}$ to $C_i$ at level $3$.   Then the third rows of these two columns contributes a factor of $c_i(r^i-r^{i_2})$ to $f_\mathscr t$. By our hypothesis  $c_i \notin \Supp f_\mathscr D$, so as above it follows that $C_i,C_{i_2}$ have height $\geq 4$ and then by \ref {3.1}(iii) the entry in the fourth row of $C_{i_2}$ is again $i$ giving a factor which cancels with the previous one.

  Continuing by induction on $m=2,3,\ldots,$ it follows as above that if $s \geq 2m-1$, the entry in the $(2m-1)^{th}$ row of $C_i$ is $i$ and then that $C_i$ has a right neighbour $C_{i_m}$ at levels $2m-1,2m$ with entry $i$ in its $2m^{th}$ row. In particular $s \geq 2m$.

   We conclude that $s$ is even and that $C_i$ has a right neighbour at levels $s-1,s$ of height $\geq s$.  Thus a left even domino may be removed from $C_i$.  This contradicts the hypothesis that $\mathscr T$ is deplete.

 \end {proof}

\subsection{Class Separation}\label{4.7}

Let $\mathscr D$ be a non-empty diagram of order $t+1$ with an empty column $C_i$.  By the right boundary condition $i\leq t$. Removal of this column gives a diagram $\mathscr D'$ of order $t$ which still satisfies the boundary conditions.  Via \ref {3.1}, the corresponding tableau $\mathscr T$ and $\mathscr T'$ also coincide except that the set $T$ should be replaced by $T^-:=T\setminus \{i\}$.  Clearly if $\mathscr T$ is deplete, then so is $\mathscr T^*$.  We call this procedure collapsing an empty column.

\begin {lemma} The map $\mathscr T \mapsto f_\mathscr T$ restricted to deplete tableaux is injective.
\end {lemma}

\begin {proof}

 By Lemma \ref {4.6} we only have to prove the assertion for tableaux with the same support. The assertion is trivial for $t=1$. If the support of $\mathscr T, \mathscr T'$ is not $T$, then the assertion holds through induction hypothesis on $t$ by collapsing an empty column.  Otherwise we proceed as follows.

Suppose $f_\mathscr T=f_{\mathscr T'}$. Then $f_{\mathscr T^*}=f_\mathscr T^*=f_{\mathscr T'}^*=f_{\mathscr T'^*}$.  On the other hand if $\Supp\mathscr T=\Supp\mathscr T' =T$,   then $\Supp \mathscr T^*=\Supp \mathscr T'^* \varsubsetneq T$. Thus by the first part $\mathscr T^*=\mathscr T'^*$ and so $\mathscr T = \mathscr T'$, as required.

\end {proof}

\subsection{}\label{4.8}

Recall \ref {2.4}.  In view of Lemmas \ref {4.5} and \ref {4.7} it makes sense to write $f_\mathscr T \in H$ whenever $\mathscr T \in H$.

  \section{Complete Tableaux}\label{5}

  Throughout this section we view the $c_i:i \in T$ and the $m^i:i \in \hat{T}$ as indeterminates.

\subsection{}\label{5.1}

Let $\mathscr T$ be a complete tableau.  The aim of this section is to describe how to read off $f_\mathscr T$ from the height function of $\mathscr T$.  This calculation is made easier by the fact that $\mathscr T$ is well-numbered (\ref {3.1}).

These results will allow us to show that the coefficient of $m_i$ in $f_\mathscr T$ is zero (resp. $\pm c_i$) exactly  when $C_i$ is the strongly extremal (resp. a quasi-extremal) column of $\mathscr T$.

Again suppose $\mathscr T \in H_k, \mathscr T' \in H_j$. Then we will be able to compute when
$$f_\mathscr T-f_{\mathscr T'}=c_i(r^k-r^j),\eqno{(*)}$$
for some $i \in T$ and show that one cannot have $f_\mathscr T-f_{\mathscr T'}=-c_i(r^k-r^j)$, which we call the opposite of $(*)$.

A further consequence is that for a \textit{complete} tableau $\mathscr T$ the partial order defined in \ref {3.2.1}$(**)$ (resp. \ref {3.2.1}$(*)$) determines the sectors in which $f_\mathscr T$ (resp. $f^*_\mathscr T$) satisfies the positivity condition (as defined below).  This will be used to show that the partial order on a complete tableau is weaker than the partial order defined on any tableau in its equivalence class.

%Finally as a by-product we will obtain an intrinsic proof of Lemma \ref {2.3.7} and of the first part of Lemma \ref {2.3.1}.

\subsection{}\label{5.2}

Let $\mathscr T$ be a tableau and set $u:=ht \mathscr T$ is even. Let $P_u$ denote the union of rows $R_u,R_{u-1}$.  Set $f_{P_u}=f_{R_u}+f_{R_{u-1}}$.

 Let $C_{i_1},C_{i_2},\ldots, C_{i_m}$ be the columns of $\mathscr T$ of height $u$.  Interspersed between $C_{i_j}$ and $C_{i_{j+1}}$ there are columns $C_{j_{k_1}},\ldots, C_{j_{k_n}}$ of height $u-1$. Set $j_{k_{0}}=i_j,j_{k_{n+1}}=i_{j+1}$.

 Recall that $b(j_{k_{n}},u-1)=b(j_{k_{n+1}},u)$. From \ref {4.4}$(*)$ we obtain

$$\sum_{\ell=0}^{n} c_{b(j_{k_\ell},u-1)}(r^{j_{k_{\ell}}}-r^{j_{k_{\ell+1}}})+c_{b(j_{k_{n+1}},u))}(r^{i_{j+1}}-r^{i_{j}})\quad \quad $$
$$\quad \quad \quad \quad \quad  =\sum_{\ell=0}^{n-1} (c_{b(j_{k_\ell},u-1)}-c_{b(j_{k_n},u)})(r^{j_{k_{\ell}}}-r^{j_{k_{\ell+1}}}).$$

The condition of positivity of this expression is that $c_{b(j_{k_\ell},u-1)}\geq c_{b(j_{k_n},u)}$, for all $\ell=1,2, \ldots,n-1$.

Of course a similar result holds for all $j=1,2,\ldots,m-1$.  In the case $j=m$ we obtain instead the expression
$$\sum_{\ell=0}^{n-1}c_{b(j_{k_\ell},u-1)}(r^{j_{k_{\ell}}}-r^{j_{k_{\ell+1}}}).$$

The overall contribution  to $f_{P_u}$ is the sum of all these expressions.

 Then the condition for the positivity of $f_{P_u}$ is just the set inequalities which result from the partial order given in   \ref {3.2.1}$(*)$.

 The expression for $f_\mathscr T$ is the sum of $f_{P_u}+f_{P_{u-2}}+\ldots+f_{P_2}$.  However since there may be cancellations, the condition for positivity of $f_\mathscr T$ can be \textit{ weaker } than that expressed by the combined set of inequalities above.

 Now suppose that $\mathscr T$ is complete. Recall (Lemma \ref {2.3.2}) that $ht C_1=u$ and $ht C_{t+1}\geq u-1$.
 By Lemma \ref {2.3.3}, the height function for a complete tableau implies that $j_{k_n}=i_{j+1}-1$. Since $\mathscr T$ is well-numbered the above equations become

 $$\sum_{\ell=0}^{n-1} (c_{j_{k_\ell}}-c_{j_{k_n}})(r^{j_{k_{\ell}}}-r^{j_{k_{\ell+1}}}).\eqno{(*)}$$

$$\sum_{\ell=0}^{n-1} c_{j_{k_\ell}}(r^{j_{k_{\ell}}}-r^{j_{k_{\ell+1}}}).\eqno{(**)}$$

The set of indices occurring in $(*)$ or in $(**)$, namely $\{j_{k_0},j_{k_1},\dots,j_{k_n}\}$, consists of $j_{k_0}$ which is an index of a column of height $u$ whose neighbour to the right has height $<u$ and $\{j_{k_1},\dots,j_{k_n}\}$ which is the set of indices of columns of height $u-1$.

\
We conclude that the indices occurring in the expression for $f_{P_u}$ are amongst those which label the columns of heights $u,u-1$.

\

 It follows from the above that the indices occurring in $f_{P_u},f_{P_{u-2}}, \ldots,f_{P_2}$ are pairwise distinct.

 \

 These conclusions are unchanged when $u$ is odd.  Indeed if $u$ is odd one has
 $$f_{R_u}=\sum_{\ell=1}^{m-1}c_{i_\ell}(r^{i_\ell}-r^{i_{\ell+1}}), \eqno{(***)}$$
 and again the indices occurring in $f_{R_u},f_{P_u},f_{P_{u-2}}, \ldots,f_{P_2}$ are pairwise distinct.

 These conclusions may be summarized as follows.

 \begin {lemma} Let $\mathscr T$ be a complete tableau of height $2v$ (resp. $2v+1$).  Then the indices occurring in $f_{P_{2i}}:i=1,2,\ldots, v$ (resp. and $f_{R_{2v+1}}$) are pairwise distinct.
 \end {lemma}

 \textbf{Remark}.  This is false for a tableau which is not complete.  Take $\mathscr T =(3,1,2,3)$ which is not complete.  One checks that $f_{P_4}=f_{R_3}=c_2(r^1-r^4),f_{P_2}=(c_1-c_2)(r^1-r^2)$.

 \subsection{}\label{5.3}

Recall (\ref {4.2}) the notion of a sector.  Let $\mathscr T$ be a complete tableau.  From Lemma \ref {5.2}, it is immediate that

\

(i)  The sectors in which $f_\mathscr T$ satisfies positivity are exactly those  given by \ref {3.2.1}$(*)$ (where we recall that $j$ is even).

\

Of course a similar result holds for $\mathscr T^*$.  It gives the following result.

\

 (ii) The sectors in which $f^*_\mathscr T$ satisfies positivity are exactly obtained as being the dual sectors defined by the inequalities  given by \ref {3.2.1}$(**)$ (where we recall that $j$ is odd).

 \

We may summarize these conclusions in the following manner.  Recall the partial order $P(\mathscr T)$ defined in \ref {3.2.1}.

\begin {lemma}  Let $\mathscr T$ be a tableau and $\mathscr T^c$ be a complete tableau in its equivalence class. Then the sectors defined by $P(\mathscr T)$ are a subset of those defined by $P(\mathscr T^c)$.  Moreover the sectors defined by $P(\mathscr T^c)$ are exactly those defined by the combined positivity conditions on $f_\mathscr T$ and on $f^*_\mathscr T$ (as made more precise by (i) and (ii) above).
\end {lemma}

\textbf{Remark}.  In principle the first part can be proved by studying the behaviour of $P(\mathscr T)$ under domino adjunction.  However this behaviour seems too complicated for this calculation to be feasible.

\

\textbf{Example 1.}  Take $\mathscr T=(2,1,0,2)$ which has completion $\mathscr T^c=(2,1,2,2)$.  Then $P(\mathscr T)=\{2<1,2<3\}$, whereas $P(\mathscr T^c)=\{2<1\}$.  In fact $f_T=h+(c_1-c_2)(r^1-r^2),f^*_T=h+c_1(r^1-r^2)+c_2(r^2-r^4)$.

\

\textbf{Example 2.}  Take $\mathscr T=(4,2,1,3,4)$.  Then $f_\mathscr T= (c_1-c_4)(r^1-r^4)+(c_2-c_3)(r^3-r^4)$, which satisfies positivity in the sectors defined by the inequalities $c_1\geq c_4,c_2\geq c_3$.  Yet $f_\mathscr T$ also satisfies positivity when $c_1\geq c_4$ and $c_1+c_2\geq c_3+c_4$ which is \textit{part} of some further sectors.  This is why we speak of positivity in an\textit{ entire} sector.

\

\textbf{Definition}.  For all $\textbf{c}$, viewed as a linear order $L(\mathscr T)$ on a tableau $\mathscr T$, let $H^{t+1}(\textbf{c})$ (or simply $H(\textbf{c})$) denote the subset of $H^{t+1}$ of all complete tableaux $\mathscr T$ such that $P(\mathscr T)$ lifts to $L(\mathscr T)$.

\

 By the lemma, for all $\mathscr T \in H(\textbf{c})$, the functions $f_\mathscr T, f^*_\mathscr T$ satisfy the positivity conditions in the sector defined by $L(T)$.  Moreover we obtain, in the notation of \ref {4.3}, that
 $$(H^{t+1}(\textbf{c}))^* =H^{t+1}(\textbf{c}^*). \eqno {(*)}$$

\subsection{}\label{5.4}

Recall the conventions and notations of \ref {4.4}. Set $K_s=\{1,2,\ldots, k_s\}$.  Let $\mathscr T_s$ be the tableau formed from the first $s$ rows of $\mathscr T$.  Then $\{C_{u_i}\}_{i \in K_s}$ is the set of columns of $\mathscr T_s$ of height $s$.  Let $K_{s+1}=\{i \in K_s|B(i,s+1) \ \text{is a block}\}$.   Then $\{C_{u_i}\}_{i \in K_{s+1}}$ is the set of columns of $\mathscr T_{s+1}$ of height $s+1$.

Recall \ref {2.3.4} and \ref {3.1}.

In what follows we take $s$ odd (resp. even).  Then $C_{u_i}$ is the strongly extremal column of $\mathscr T_s$ if $i$ is the minimal (resp. maximal) element of $K_s$.  Again $B(u_i,s)$ is the extremal block of $R_s$ if $i$ is the maximal (resp. minimal) element of $K_s$.

In the notation of  \ref {4.4}, set $h^s=h+\sum_{i=1}^{s-1}f_{R_i}$.

The following result is crucial to the construction of the dual Kashiwara functions.

\begin {lemma}  Let $\mathscr T$ be a tableau and $C_i$ its unique strongly extremal column.   Then the coefficient of $m_i$ in $f_\mathscr T$ equals zero.
\end {lemma}

\begin {proof}

Retain the above conventions and notation in particular taking $s$ odd (resp. even).  Recall the notation of \ref {3.1} and let $\hat{I}_s$ denote the set of blocks of $R_s$.

In the induction step on $s$ below, it is convenient to view $\theta_s$ as a map from $I_s$ to itself by taking $R_{s+2}=R_{s+1}=R_s$ (which is permissible by \ref {2.3.6}) and then to extend $\theta_s$ to $\hat{I}_s$ with the property that $c_{\theta_s(u_i)}=0$, if $B(u_i,s+1)$ is the extremal block of $R_{s+1}$. This last condition means that $i$ is the minimal (resp. maximal) element of $K_s$.

With these conventions we claim that for all $i \in K_s$ the coefficient of $m^{u_i}$ in $h^s$ equals $(-1)^{s-1}c_{\theta_{s}(u_{i})}$. The latter equals zero when $i$ is the minimal (resp. maximal) element of $K_s$.  Then $C_i$ is just the unique strongly extremal column of $\mathscr T_s$ proving the lemma.

 One has $\theta_0(u_i)=\theta_0(i)=i$.  Thus the asserted coefficient of $m^i$ in $h^0$ equals $-c_i:i \in \hat{T}$.  Yet $h^0=h=-\sum_{i=1}^tc_im^i$, whilst $B(t+1,1)$ is an extremal block, so $c_{t+1}=0$ by convention.  Hence the assertion holds for $s=0$.

 Let $M_{s}$ be the $\mathbb Z$ module generated by the $m^j: j \in \hat{T}\setminus K_{s}$.

 Assume $s$ is even.

By the first part of \ref {4.4}$(*)$ we obtain
$$h^{s+1}-h^{s}=
(-1)^{s}\sum_{i \in K_{s+1}}c_{\theta_{s}(u_{i})}(m^{u_i}+m^{u_{i+1}}) \  \text{mod} \ M_{s+1},\eqno{(*)}$$

Then by the induction hypothesis and $(*)$ it follows that for all $i \in K_{s+1}$ the coefficient of $m^{u_i}$ in $h^{s+1}$ is $(-1)^{s}c_{\theta_{s}(u_{i-1})}$.  Yet since $s+1$ is odd and $C_{u_{i-1}}$ is the left neighbour to $C_{u_i}$ at level $s$ it follows that $\varphi_s(u_i)=u_{i-1}$ and so $\theta_{s}(u_{i-1})=\theta_{s+1}(u_i)$. Through duality or a similar argument (using the second part of \ref {4.4}$(*)$ and replacing $u_{i+1}$ by $u_{i-1}$ in the last step) gives the assertion for  $s$ odd.  This completes the induction.

\end {proof}

\subsection{}\label{5.5}

 Let $\mathscr T$ be a tableau of height $r$.

We shall now give a further (perhaps more natural) proof of Lemma \ref {5.4} using the fact that by \ref {2.3.4} it is enough to establish it for $\mathscr T$ complete and that a complete tableau is well-numbered.

 If $C_j$ is a column of $\mathscr T$ of height $r$ which is not the strongly extremal column, then the proof of Lemma \ref {5.4} shows that the coefficient of $m^j$ in $f_\mathscr T$ takes the form $\pm c_i$ for some $i \in T$ which moreover can be determined from $\mathscr T$.  It is not hard to show that this fact is also a consequence of Lemma \ref {5.4}.  Again if we further assume that $\mathscr T$ is complete, then these columns are amongst those which a quasi-extremal. We shall further show that for $\mathscr T$ \textit{complete}, a column  $C_j$ is a quasi-extremal if and only if the coefficient of $m^j$ in $f_\mathscr T$ takes the form $\pm c_i$ for some $i \in T$.  This can fail if $\mathscr T$ is not complete.  Of course for an arbitrary tableau $\mathscr T$ we can define its quasi-extremal columns by this property.  However they then become difficult to identify.

Recall the notion of single box linkage \ref {3.3}.
% View the $c_i:i \in T$ as indeterminates.

\begin {lemma} Let $\mathscr T$ be a complete tableau of height $u$ and $C_j$ its unique strongly extremal column.
 The coefficient of $m_k$ in $f_\mathscr T$ equals $\sn(k-j) c_i$, for some $i \in \{1,2,\ldots,t\}$ if and only $C_k$ is a quasi-extreme column of $\mathscr T$.

\end {lemma}

\begin {proof} Assume $u$ is odd (resp. even). Recall that single block linkage adds a block to a quasi-extreme column $C_k$ of $\mathscr T$ of height $u-1$ to give a new tableau $\mathscr T \cup B$ with strongly extremal column $C_k$. Moreover (\ref {3.3.1}) one has $k<j$ (resp. $k>j$).  From \ref {4.4}$(**)$ we obtain
$$f_{\mathscr T \cup B}-f_\mathscr T=c_i(r^k-r^j), \eqno{(*)}$$
where $i$ may be read off from $\mathscr T$.  Observe that $(*)$ is just what we promised to prove in \ref {5.1} and that we do not obtain the opposite of \ref {5.1} $(*)$.  This is a crucial fact and was not a priori obvious.

Observe that $r^k-r^j= \sn(j-k)(m^k+m^j)$.  By Lemma \ref {5.4} (of which we shall also give an independent proof below) the coefficient of $m_k$ in $f_{\mathscr T \cup B}$ is zero.  Then by $(*)$ it follows that the coefficient of $m_k$ in $T$ must be $\sn (k-j)$.  This proves ``if'' of the lemma.

The formula for $f_\mathscr T$ is given by adding $h=-\sum_{i=1}^tc_im^i$, to the sum $S$ of all the expressions in $(*),(**), (***)$ of \ref {5.2}. By Lemma \ref {5.3} the only coincidences in the coefficients of $m^i: i \in \hat{T}$ occur between the formula for $h$ and one of the said expressions.  Since we are treating the $c_i:i \in T$  as indeterminates  the latter can give rise to no cancellations amongst themselves. To avoid cancellations we may also restrict the $c_i:i \in T$ to be strictly positive integers which are strictly decreasing and in ``general position''. Then the $c_j$ and the $c_{j_{k_\ell}},c_{j_{k_\ell}}-c_{j_n}$ are strictly positive.

Finally observe that
$$r^{j_{k_\ell}}-r^{j_{k_{\ell+1}}}=
m^{j_{k_\ell}}+2m^{j_{k_\ell}+1}+\ldots+2m^{j_{k_\ell+1}-1}+m^{j_{k_{\ell+1}}}.$$

 Thus to obtain a coefficient of $\pm c_i$ or $0$ of $m_j$ in $f_\mathscr T$ we must have a coefficient of $2c_i,c_i,0$ of $m_j$ in $S$.  This means that $C_j$ cannot lie between columns of strictly greater height.  By Lemma \ref {2.3.2} this implies that $ht C_j \geq u-1$.

 Suppose $ht C_j=u$ and that $u$ is odd.

 By \ref {5.2}$(***)$, the coefficient of $m^{i_\ell}$ equals $c_{i_\ell}$ if $\ell=1$ or if $\ell=  m$. The coefficient of $m^{i_\ell}$ equals $2c_\ell$ if $1<\ell<m$. Moreover $i_m=t+1$, by Lemma \ref {2.3.2}. Moreover because indices on the $r^j$ are pairwise distinct with respect to the top row and pairs of subsequent rows, this remains true in the overall sum $S$. Thus the coefficient of $m^{i_\ell}$ in $f_\mathscr T$ is $0$ if $\ell=1$ (which defines the strongly extremal column) and the coefficient of $m^{i_\ell}$ equals  $c_{i_\ell}$ if $1<\ell \leq m$ (which defines the quasi-extremal columns of height $u-1$).

 The case $ht C_j=u$ and $u$ even, results by duality.

One may check (for $\mathscr T$ complete) that this gives a second proof of the assertion in the induction hypothesis of Lemma \ref {5.4} and in particular of Lemma {5.4}.

Suppose $ht C_j=u-1$ and $u$ odd.  By the above we need only consider the columns of height $u-1$ which do not lie between columns of height $u$ and hence lie to the left the strongly extremal column $C_{i_1}$.  It is then enough to consider the case $1<i_1$.  In this case we adjoin a half-domino (see \ref {2.3.4}) to the strongly extremal column $C_{i_1}$ and complete to a tableau $\hat{T}$ of height $u+1$.  By Lemma \ref {2.3.4} the columns of height $u-1$ to the left of $C_{i_1}$ become of height $u+1$.  Then the assertion for this case reduces by duality to the first case.  The case $u$ even results by duality (or by a similar argument).

\end {proof}

\subsection{}\label{5.6}

We now prove a version of Lemma \ref {4.5} for complete tableaux, by a method which is independent of Lemma \ref {2.3.7}.
%This will give a second more intrinsic proof of Lemma \ref {2.3.1}.

 \begin {lemma} Let $\mathscr T$ and $\mathscr T'$ be complete tableau of the same height $u$.  If $f_\mathscr T=f_{\mathscr T'}$, then $\mathscr T=\mathscr T'$.
 \end {lemma}

 \begin {proof}  Suppose $u=2v$ is even. By $(*),(**)$ of \ref {5.2} we can read off the columns of height $2i,2i-1$ of $\mathscr T$ from $f_{P_{2i}}$. Yet the hypothesis and Lemma \ref {5.2} imply that $f_{P_{2i}}
=f_{P_{2i}'}$, for all $i=1,2,\ldots,v$. Hence the assertion.  The case where $u$ is odd obtain by duality or by a similar argument.
 \end {proof}

 \subsection{}\label{5.7}

 Define an order relation on the set $f_\mathscr T : \mathscr T \in H^{t+1}(\textbf{c})$, by $f_\mathscr T \geq f_\mathscr T'$ if $f_\mathscr T-f_\mathscr T' \in \sum_{i=1}^t\mathbb N(r^i-r^{i+1})$.

 %It is clear that $z$ is the unique minimal element of $\{f_\mathscr T:\mathscr T \in H^{n+1}(\textbf{c})\}$.  Let us define $h^j$ for all $j =1,2,\ldots n+1$ inductively by $h^{n+1}=z, \quad h^j= c_j(r^j-r^{j+1})+h^{j+1}$.  These elements form the vertices of the pointed chain \cite [6.3]{J2}, in $H^{n+1}$ and belong to $H^{n+1}(\textbf{c})$ for all linear orders $\textbf{c}$.  One has $h^j \in H^{n+1}_j$.

\begin {lemma}  $h$ (resp. $h^*$)  is the unique minimal (resp. maximal) element of $\{f_\mathscr T:\mathscr T \in H^{t+1}(\textbf{c})\}$.
\end {lemma}

\begin {proof}  This is an immediate consequence of the definition of $H^{t+1}(\textbf{c})$ of $(*)$ of \ref {5.3}.
%of duality \cite [4.3]{J2}.  This is an involution $*$ on $\{f_\mathscr T:\mathscr T \in H^{n+1}\}$, defined by $c_i^*=c_{n+1-i}, (r^i)^*=r^{n+2-i}$ and linearity. It takes $h^{n+1}$ to $h^1$.  Moreover
%$$(H^{n+1}(\textbf{c}))^* =H^{n+1}(\textbf{c}^*), \eqno {(*)}$$
%by construction (see Remark).  Hence the assertion.
\end {proof}

\textbf{Remark}. The condition that $f \in H^{t+1}$ belongs to $H^{t+1}(\textbf{c})$ is that it satisfies certain positivity conditions with respect to $\textbf{c}$ (which implies that $h$ is the unique minimal element of $f_\mathscr T : \mathscr T \in H^{n+1}(\textbf{c})$) \textit{and} that $f^*$ satisfies corresponding positivity conditions with respect to $\textbf{c}^*$.  The latter were imposed for no particular reason except that they (mysteriously) led to the required invariance property of $f_\mathscr T : \mathscr T \in H^{t+1}(\textbf{c})$ implied by it forming an $S$-graph.  However imposing these latter conditions is exactly what is needed to obtain that $h^*$ is the unique maximal element of $H^{t+1}(\textbf{c})$ above.  Then the truth of the above lemma gives a more cogent reason for adopting these additional conditions.  Eventually we would like to think of duality as a reflection corresponding to the simple root $\alpha$ with respect to which the Kashiwara functions $r^i:i \in \hat{T}$ as defined (cf \ref {4.2}).

 \section{The Graph of Links of $H^{t+1}$}\label{6}

\subsection{}\label{6.1}

Recall (\ref {3.3.3}) the (undirected) graph $\mathscr G_{H^{t+1}}$ (or simply, $\mathscr G_H$) of links of $H^{t+1}$. Let $\mathscr G_{H(\textbf{c})}$ be the subgraph of $\mathscr G_H$ whose vertices are $H(\textbf{c})$ and edges are links between the elements of $H(\textbf{c})$.  Here we establish some properties of $\mathscr G_H$ and of $\mathscr G_{H(\textbf{c})}$.

 More generally let  $\mathscr G$ be a graph on which we shall want to impose a number of properties, namely: $P_1-P_7$ below. In this the set of vertices (resp. edges) of $\mathscr G$ will be denoted by $V( \mathscr G)$ (resp. $E(\mathscr G)$).  An edge joining the vertices $v,v' \in V( \mathscr G)$ will be denoted by $(v,v')$.

\

$(P_1)$.  There is map $i \mapsto i_v$ of $V(\mathscr G)$ into $\hat{T}$.  One has $i_v\neq i_{v'}$ if $(v,v') \in E(\mathscr G)$.

\

$(P_2)$.  There is map $i \mapsto i_{(v,v')}$ of $E(\mathscr G)$ into $T$.  One has $i_{(v,v')}\neq i_{(v,v'')}$, if $v' \neq v''$.

\

In \ref {3.3.3}, we showed these to hold for $\mathscr G_H$.

\subsection{Evaluation}\label{6.2}

From now on all graphs will be assume to satisfy $P_1,P_2$.  We write $c_{i_{(v,v')}}$ simply as $c_{v,v'}$ or as $c_{v',v}$.

%Let $\hat {H}^{t+1}$ (or simply, $\hat{H}$)

A graph $\mathscr G$ will be said to be an evaluation graph if

$(P_3)$.  There exists a map $f:v \mapsto f_v \in \hat{H}$ such that
$$f_v-f_{v'}=c_{v,v'}(r^{i_v}-r^{i_{v'}}).$$

\

%Even we do not take the target set to be $H$ this condition is non-trivial for graphs which are not trees.
%
We have shown (see \ref {5.5}$(*)$) that $\mathscr G_H$ satisfies $P_3$.

\subsection{The Pointed Chain}\label{6.3}

A pointed chain $\mathscr C$ in  $\mathscr G$ is a set $V(\mathscr  C) =\{v_{t+1},v_t,\ldots,v_1\}$ of vertices of $\mathscr G$ with $i_{v_i}=i$ for all $i \in \hat{T}$ and $i_{(v_{i+1},v_i)}=i$, for all $i \in T$.

\

$(P_4)$.  There is a unique pointed chain $\mathscr C$ in $\mathscr G$.

\

\begin {lemma}  $\mathscr G_H$ satisfies $P_4$. Moreover for every linear order $\textbf{c}$ one has $\{f_v: v \in V(\mathscr C)\} \subset H(\textbf{c})$ and so in particular $ H(\textbf{c})$ satisfies $P_4$.

\end {lemma}

 \begin {proof}

Existence.  Let $\mathscr T_{t+1}$ be the empty diagram (or tableau).  For all $i \in T$, let $\mathscr T_i$ be the tableau of height $1$ with $C_j:j=i,i+1,\ldots,t+1$ being its non-empty columns (of height $1$). Set $h_i=f_{\mathscr T_i}:i \in \hat{T}$. One has $h_i \in H_i$.  Let $v_i $ be the corresponding vertex in $V(\mathscr G_H)$.  One checks that $v_{t+1},v_t,\ldots,v_1$ is a pointed chain in $\mathscr G_H$.

Uniqueness.    Let $f_{t+1},f_t,\ldots,f_1 \in H$ be the functions corresponding to the successive elements $v_{t+1},v_t,\ldots,v_1$ in a pointed chain. In particular $f_{t+1} \in H_{t+1}$. By $P_2$ it is enough to show that $f_{t+1} =h$. Assume this is not true.   Let $\mathscr T_i:i \in \hat{T}$ be the unique complete tableau of minimal height corresponding to $f_{t+1}$. By Lemma \ref {5.4} it follows that $C_{t+1}$ is the strongly extremal column of $\mathscr T_{t+1}$ and hence its height $s$ is even.  Since $\mathscr T_{t+1}$ is complete, the top entry of this column is $c_t$. Since $f_{t+1}-f_t=c_t(r^{t+1}-r^t)$, by the assumption on the chain, it follows that the second to last column $C_t$ of $\mathscr T_{t+1}$ also has height $s$  and that $\mathscr T_t$ is obtained from $\mathscr T_{t+1}$ by deleting the block from $C_{t+1}$ at level $s$.  Applying this argument successively to $\mathscr T_i:i=t+1,t,t-1,\ldots,1$, it follows that the top row at level $s$ of $\mathscr T_{t+1}$ has cardinality $t+1$.  This contradicts the assumption that $f_{t+1} \neq h$ and that $\mathscr T_{t+1}$ is reduced (or without these hypotheses it simply means that $f_{t+1}=h$).

Since $\mathscr T_t$ has height $\leq 1$, is complete and satisfies $P(\mathscr T_t=\phi$ because there are gaps in the columns.   Hence the last part.

\end {proof}

\subsection{The Marked Vertex}\label{6.4}

We shall say that a graph $\mathscr G$ is pointed if it satisfies $P_4$. In a pointed graph $\mathscr G$ we identify the initial vertex of the pointed chain with $h$ and we denote this vertex by $v_h$. It is called the marked vertex. Then if $P_3,P_5$ also hold we can compute $f_v$ for all $v \in V(\mathscr G)$.

\subsection{Connectedness}\label{6.5}

\

$(P_5)$.  $\mathscr G$ is connected.

\

In view of Lemma \ref {6.3}, it is enough to prove the stronger result which we also need in \ref {8.5}.

\begin {lemma} For every linear order $\textbf{c}$, the subgraph $G_{H(\textbf{c})}$ is connected.
\end {lemma}

\begin {proof}  It is enough to show that there is a path from any vertex $v \in V(\mathscr G_{H(\textbf{c})})$ to $v_h$.  Take the unique complete tableau $\mathscr T \in H(\textbf{c})$ of minimal height $u$ representing $v$.

In what follows we take $u$ odd (resp. even).

If $\mathscr T$ is the empty tableau we are done.  Otherwise we prove the assertion by induction on the number of blocks of $\mathscr T$. Let $C_j$ be the unique strongly extremal column of $\mathscr T$.  Then $B(j,u)$ cannot be the only block in $R_u$ for otherwise $\mathscr T$ would not be of minimal height (c.f. \ref {2.3.4}). In particular it is not the extremal block $B(k,u)$ of $R_u$.  (One may wish to recall that $j$ is minimal (resp. maximal) such that $B(j,u)$ is a block of $\mathscr T$, which just the opposite is true of $k$.) Though rule (iii) of \ref {3.1} one has $i:=b(j,u)\in T$.  Let $\mathscr T \setminus B(j,u)$ be the tableau obtained by removing $B(j,u)$.  from $\mathscr T$.  The resulting diagram obviously satisfies the boundary conditions at all levels except $u$ and for level $u$ this also hold because $B(j,u)$ is not the extremal block at level $u$.  It is immediate from definitions (\ref {3.2.1}) that $P(\mathscr T \setminus B(j,u))\subset P(\mathscr T)$.  Hence $\mathscr T \setminus B(j,u)) \in H(\textbf{c})$  by Lemma \ref {5.3}.  Obviously $\mathscr T \setminus B(j,u))$ is complete and therefore cannot represent the same element of $H$ as $\mathscr T$.

However besides this one easily checks that the unique strongly extremal column of $\mathscr T \setminus B(j,u))$ is $C_\ell$ with $j<\ell \leq k$ (resp. $k\leq \ell <j$) and that $f_\mathscr T -f_{\mathscr T \setminus B(j,u)}=c_i(r^j-r^\ell)$.
\end {proof}

\subsection{Triads}\label{6.6}

A triad in a graph $\mathscr G$ is a set of four vertices $(a,b,c,d)$ successively joined by three edges $(a,b),(b,c),(c,d)$ such that $i_{a,b}=i_{c,d}$.

A graph is called triadic if for every triad $(a,b,c,d)$ one has $i_a=i_d$.

The significance of this condition obtains from $P_3$ since it implies that
$$f_a-f_d=(c_{b,c}-c_{a,b})(r^{i_c}-r^{i_b}). \eqno {(*)}$$

\

$(P_6)$.  $\mathscr G$ is triadic.

\

For every triad in the graph $\mathscr G_H$ fix complete tableaux $\mathscr T_a,\mathscr  T_b,\mathscr T_c,\mathscr T_d$ representing $f_a,f_b,f_c,f_d$ having heights $u_a,u_b,u_c,u_d$.

\begin {lemma}  $\mathscr G_H$ satisfies $P_6$. Moreover the inequality $c_{b,c}>c_{a,b}$ results from either $\mathscr T_a$ or $\mathscr T_d$ but not both.
\end {lemma}

\begin {proof}

  By the construction of links (see \ref {3.3.2} and Remark) we can assume that $u_c=u_b=:u$. Then by possibly adding a half-domino to the extremal columns of $\mathscr T_b,\mathscr T_c$ and completing we can further assume that $u_d=u$ We can also assume that $u\geq 2$ by adding two complete rows to all the tableaux (alternatively we can adjoin a zeroth row to every tableau with $i-1$ in $B(i,0):i>1$ with a diagonal slash in $(1,0)$ - all this being essentially a matter of book-keeping).

  Suppose in what follows that $u$ is odd (resp. even). Then $i_d<i_c<i_b$ (resp. $i_d>i_c>i_b$).

  Let $\overline{\mathscr T}_c$ be the tableau obtained from the complete tableau $\mathscr T_b$ by adjoining a block $B$ to the column  $C_{i_c}$ of height $u-1$ in $\mathscr T_b$ and let $c_i$ be the entry in the top row of the first  column $C_m$ of height $\geq u-1$ on passing from $C_{i_c}$ to $C_{i_b}$ in $\mathscr T_b$. By \ref {3.1}, the entry of $B$ must be $c_i$.

  Let $\overline{\mathscr T}_d$ be the tableau obtained from the complete tableau $\mathscr T_c$ by adjoining a block $B$ to the column $C_{i_d}$ of height $u-1$ in $\mathscr T_c$ and let $c_j$ be the entry of the top row of the first column $C_n$ on passing from $C_{i_d}$ to $C_{i_c}$ in $\mathscr T_c$ of height $\geq u-1$. By \ref {3.1}, the entry of $B$ must be $c_j$.

  Consequently we can obtain a tableau $\overline{\mathscr T}_a$ from $\mathscr T_b$ by placing a block containing $c_j$ on the top of $C_{i_d}$ (hence in $R_u$) with the property that $f_b-f_a=c_j(r^{i_b}-r^{i_a})$.  Through $P_2$ this forces $i_d=i_a$.  Furthermore in $\mathscr T_a$ the column $C_{n-1}$ (resp. $C_{n+1}$) has height $u$ and its block in $R_u$ contains $c_j$.  On the other hand there are no columns in $\mathscr T_a$ of height $u$ between and including $C_n$ and $C_m$, whilst $c_i$ is the entry in the top row of $C_m$.  Thus the inequality $c_j<c_i$ is realised by $\mathscr T_a$.  It is not realized by $\mathscr T_d$ because the passage from $\mathscr T_a$ to $\mathscr T_d$ is the adding of a block with entry $c_i$ to $C_{i_c}$ which becomes of height $u$ and lies between $C_n$ and $C_m$.  This completes the proof of the lemma.

  One may also check that the last operation of adding a block \textit{between} two columns of height $u$ on passing from $\mathscr T_a$ to $\mathscr T_d$ gives
  $$f_a=f_d+ (c_i-c_j)(r^{i_c}-r^{i_b}),$$
  which agrees as it should with $(*)$.

\end {proof}

\subsection{$S$-graphs}\label{6.7}

Let $\mathscr G$ be a triadic graph.  To a triad $(a,b,c,d)$ of $\mathscr G$ we define the order relation $i_{a,b}<i_{b,c}$. Let $P(\mathscr G)$ be the partial order on $T$ defined be the set of all triads of $\mathscr G$. We say that $\mathscr G$ is well-ordered if $P(\mathscr G)$ can be lifted to a linear order $L(\mathscr G)$, not necessarily unique.  (In other words there are no cycles.)

Suppose that $\mathscr G$ is a triadic graph which is well-ordered and fix a linear order $L(\mathscr G)$ lifting $P(\mathscr G)$.  An ordered path from a vertex $v' \in V(\mathscr G)$ to a vertex $v \in \mathscr G$ is a sequence of linked vertices $v'=v_1 \rightarrow v_2 \rightarrow \ldots \rightarrow v_n=v$ such that $i_{v_{i-1},v_{i}} < i_{v_i,v_{i+1}}$, for all $i=2,\ldots,n-1$.  One may remark that this condition is empty unless $n>2$, so this is weaker than saying that $v$ is a source and $v'$ is a sink relative to $L(\mathscr G)$.

The $S$ condition is that for all $v' \in \mathscr G$ and all $k \in \hat{T}$ there exists $v \in V(\mathscr G)$ with $i_v=k$, such that there exists an ordered path from $v'$ to $v$.

\

$(P_7)$.   $\mathscr G$ satisfies the $S$ condition.

\

We shall say that $\mathscr G$ is an $S$-graph if it satisfies conditions $P_1-P_7$.

Though it does not concern the present work we give a brief explanation of the $S$ condition.  Assume that $P_3$ holds.  Then setting $c_{i_j}=c_{v_j,v_{j+1}}, u_j=i_{v_j}$ we may write

$$f_{v'}-f_{v}= \sum_{j=1}^{n-1}c_{i_j}(r^{u_j}-r^{u_{j+1}}),$$
where we recall that the $c_{i_j}:j=1,2,\ldots,n-1$ are integer $\geq 0$ and increasing.  Set $c_{i_0}=0$.

Thus first of all we have a sum, hence the epithet $S$.  Less trivially, since $u_n=i_v=k$, this sum can be rearranged as
$$f_v-f_{v'}= \sum_{j=1}^{n-1}(c_{i_{j}}-c_{i_{j-1}})(r^k-r^{u_j}).\eqno{(*)}$$

 This expresses $f_v-f_{v'}$ as a linear combination of the difference of the Kashiwara functions $r^k-r^{v_j}$ with \textit{non-negative integer coefficients}.  Now recalling the discussion in \ref {4.2} suppose $e_\alpha$ ``enters at the $k^{th}$ place''.  We recall (in the notation and conventions of \cite [2.3.2]{J5}) that this occurs exactly when
$r^k-r^j\geq 0:k > j$ and $r^k-r^j>0:k<j$.  Moreover $m_k$ decreases by $1$, whilst $r^k-r^j$ increases (resp. decreases) by $1$ when $k>j$ (resp. $k<j$).

At this juncture we introduce a further condition namely

\

$(P_8)$.  For all $v \in V(\mathscr G)$ the coefficient of $m_k$ in $f_v$ is zero given that $i_v=k$.

\

By Lemma \ref {5.4} this holds for $\mathscr G_H$ and hence for any of its subgraphs.

Given that $P_8$ holds we conclude from $(*)$ that when $e_\alpha$ enters at the $k^{th}$ place, $f_v-f_{v'}$ remains $\geq 0$ whilst $f_v$ is unchanged.  In particular $\Max_{v \in V(\mathscr G)}f_v$ is unchanged.

Similarly $f_\alpha$ enters at the $k^{th}$ place  exactly when
$r^k-r^j > 0:k > j$ and $r^k-r^j\geq 0:k<j$.  Moreover $m_k$ increases by $1$, whilst $r^k-r^j$ decreases (resp. increases) by $1$ when $k>j$ (resp. $k<j$). Then through $P_8$ we conclude that when $f_\alpha$ enters at the $k^{th}$ place $\Max_{v \in V(\mathscr G)}f_v$ is unchanged.

Notice that $P_3,P_7,P_8$ uniformly take care of invariance under both $e_\alpha$ and $f_\alpha$ action.  This is particularly satisfying.

%as a consequence of $(*)$, $\Max_{v \in V(\mathscr G)}f_v$  remains unchanged when $f_\alpha$ enters at the $k^{th}$ place (using the rules given in say \cite [2.3.2]{J5}). This property ensures the required invariance as expressed through (iii) of \ref {4.2}.

The above observation is the basis behind our computation of the dual Kashiwara functions.  Satisfying condition $S$ is a very non-trivial step.  In particular we found it to be practically impossible to achieve using just tableaux (see \ref {8.7}).  Yet through the graph of links it is attained by the relatively easy construction described in the next section.

Since it is reasonable to avoid the use of too many functions it is natural to impose a further condition

\

$(P_9)$.  The map $v\mapsto f_v$ separates the points of $V(\mathscr G)$.

\

By Lemma \ref {4.7} this holds for $\mathscr G_H$ and any of its subgraphs.

\subsection{Duality}\label{6.8}

  Let $\mathscr G$ be a graph satisfying $P_1,P_2$. We may define the dual graph $\mathscr G^*$ by replacing the labels on the vertices (resp edges) by the rule $i_v \mapsto t+2-i_v$ (resp. $\mapsto i_{v,v'} = t+1-i_{v,v'}$), It is immediate that if $\mathscr G$ is pointed then so is $\mathscr G^*$ since the same chain serves for both.  Again if $\mathscr G$ is triadic (resp. an $S$-graph) then so is $\mathscr G^*$.

\section{An Inductive Construction of $S$-Graphs}\label{7}

\subsection{}\label{7.1}

Fix a linear order $\textbf{c}$ on $\{c_i\}_{i \in T}$.  In this section we construct a graph $\mathscr G(\textbf{c})$ by induction on $|T|=t$. We show $\mathscr G(\textbf{c})$ is an $S$-graph and that the partial order defined by the triads lift to the given linear order $\textbf{c}$.

Recall that $|T|=t$.  When $t=1$, there can be just one graph satisfying $P_1,P_2,P_5$, namely the graph with two vertices labelled by $1,2$ respectively and joined by an edge labelled by $1$.  Since it is a tree it satisfies $P_3$.  Condition $P_4$ is also immediate, whilst $P_6,P_7$ are empty.

\subsection{Binary Fusion}\label{7.2}

 Let $c_s$ be the unique largest element of $\textbf{c}$. Then $\textbf{c}$ induces a linear order on $\textbf{c}_s^-:=\textbf{c}\setminus \{c_s\}$ (or simply $\textbf{c}^-$).
 %Since $|\textbf{c}^-|=t-1$ it is natural to use the labelling
 Set
% $$c_j^-=c_j:j\in T\setminus \{s\}.$$

 $$c^-_j= \left\{\begin{array}{ll}c_j&:j< s, \\
   c_{j-1}&:t\geq j > s. \\

\end{array}\right.\eqno {(1)}$$
Then $\textbf{c}^-$ is labelled by $T\setminus \{t\}$.

%Correspondingly the Kashiwara functions are labelled by $\hat{T}\setminus \{t+1\}$.

 Assume that we have constructed an $S$-graph $\mathscr G(\textbf{c}^-)$, where the edges are labelled by $T\setminus \{t\}$ and correspondingly the vertices are labelled by  $\hat{T}\setminus \{t+1\}$.
 %Note that $P_7$ implies in particular that  the map $v\mapsto i_v$ with target $\mathscr G(\textbf{c}^-)$ has image $\hat{T}\setminus s$.

 Let $\psi_s$ (resp. $\psi^s$) be the map $T\setminus \{t\} \rightarrow T$ (resp.  $\hat{T}\setminus \{t+1\} \rightarrow \hat{T}$) defined by

$$j\mapsto \left\{\begin{array}{ll}j& :j< s, \\
   j+1&:j\geq s. \\

\end{array}\right.\eqno {(2)}$$

%In particular $\psi_s$ implements $(1)$, whilst $\psi^s$ reverses the corresponding relabelling of the Kashiwara functions.  This triviality is just a matter of good book-keeping whose need becomes more apparent in \ref {8.1} and \ref {8.2}.

 Let  $\Psi^+_s:\mathscr G(\textbf{c}^-)\iso \mathscr G^+$ (resp. $\Psi^-_s:\mathscr G(\textbf{c}^-)\iso \mathscr G^-$) be the isomorphism of unlabelled graphs which is $\psi_s$ for the labelling on edges and $\psi^s$ (resp. $\psi^{s+1}$) for the labelling on vertices.  In particular $c_s$ does not appear as a label on any edge and $r^s$ (resp. $r^{s+1}$) does not appear on any vertex of $\mathscr G^+$ (resp. $\mathscr G^-$).

 Then $\Psi^\pm_s$ induce an unlabelled graph isomorphism $\varphi: \Psi_s^+(\mathscr G(\textbf{c}^- ))\iso  \Psi_s^-(\mathscr G(\textbf{c}^-))$ with the property that
$$i_{(v,v')}=i_{(\varphi(v),\varphi(v'))},\eqno {(3)}$$
and
$$i_v=\left\{\begin{array}{ll}i_{\varphi(v)}& :i_{\varphi(v)}\neq s, \\
   i_{\varphi(v)}+1&:i_{\varphi(v)}=s. \\

\end{array}\right.\eqno {(4)}$$

We join the above two graphs by adding an edge with index $s$ joining every vertex in $v \in \mathscr G^+$ with $i_v=s+1$ to the vertex $\varphi(v) \in \ \mathscr G^-$.  Let $\mathscr G(\textbf{c})$ denote the resulting graph.  It is clear that $\mathscr G(\textbf{c})$ satisfies $P_1,P_2,P_5$.

\subsection{}\label{7.3}

Let write $\Psi^\pm_s$ simply as $\Psi^\pm$.

\begin {lemma}  $\mathscr G(\textbf{c})$ satisfies $P_4$.
\end {lemma}

\begin {proof}

 Let $(v_1,v_2, \ldots, v_t)$ be the unique pointed chain in $\mathscr G(\textbf{c}^-)$, where we recall that $i_{v_i}=i:i=1,2,\ldots,t$ and $i_{v_i,v_{i+1}}=i:i=1,2,\ldots,t-1$.

 By the construction  $\Psi^-(v_1),\ldots,\Psi^-(v_s),\Psi^+(v_s),\ldots,\Psi^+(v_t)$ is a pointed chain in $\mathscr G(\textbf{c})$.

 Conversely let $\mathscr C:=(w_1,w_2,\ldots w_{t+1})$ be a pointed chain in $\mathscr G(\textbf{c})$.  Then $i_{w_{s+1}}=s+1$ forces $w_{s+1}\in \mathscr G^+$, and consequently the construction gives $w_s=\varphi(w_{s+1})$. Take $i\leq s$ (resp. $i>s$). Since $\mathscr C$ is a pointed chain, the construction then forces $w_i \in \mathscr G^-$ (resp. $w_i \in \mathscr G^+$).  We can therefore write $w_i=\Psi^-(v_i')$ (resp. $w_i=\Psi^+(v'_{i-1})$) for some $\{v_1',v_2',\ldots,v_t'\} \subset V(\mathscr G(\textbf{c}^-))$.  Moreover through the resulting indices on vertices and edges, the latter must define the unique pointed in $\mathscr G(\textbf{c}^-)$, that is $v_i'=v_i$, for all $i=1,2,\ldots,t$.

 \end {proof}
\subsection{}\label{7.4}

\begin {lemma}  $\mathscr G(\textbf{c})$ satisfies $P_6$.  Moreover the partial order defined by the triads lifts to the given partial order $\textbf{c}$.
\end {lemma}

\begin {proof}  By the induction hypothesis it suffices to consider a triad $(a,b,c,d)$ not lying entirely in either $\mathscr G^+$ or in $\mathscr G^-$.  Suppose $i_{a,b}=i_{c,d}=s$.  Then by construction $i_b,i_c \in \{s,s+1\}$  By $P_1$ we must have $i_b\neq i_c$, so by construction we must have $i_{b,c}=s$ contradicting $P_2$.  Thus one must have $i_{b,c}=s$ and $i_{a,b}=i_{c,d}\in T\setminus \{s\}$. Then $a,b \in \mathscr G^+$ and $c,d \in \mathscr G^-$ (or vice-versa).  This implies that $c=\varphi(b)$.  Then $P_2$ forces $d=\varphi(a)$ with $i_d\neq s,i_a \neq s+1$.  Consequently $i_d=i_a$.  This proves the first part.  Finally the relations given by these new triads are \textit{amongst} those of the form $j<s: j \in T \setminus \{s\}$, proving the second part.

\end {proof}

\subsection{}\label{7.5}

 \begin {lemma}  $\mathscr G(\textbf{c})$ satisfies $P_7$.
\end {lemma}

  \begin {proof}
  Consider a vertex $v \in \mathscr G(\textbf{c})$.  Take $v \in \mathscr G^+$ (resp. $v \in \mathscr G^-$). If $i \neq s$ (resp. $i \neq s+1$) since $\mathscr G^+$ (resp. $\mathscr G^-$) is an $S$ graph) there exists $v' \in \mathscr G^+$ (resp. $v' \in \mathscr G^-$ and an ordered path from $v$ to $v'$.  Otherwise we take $i_{v'}=s+1$ (resp. $i_{v'}=s$) and concatenate the ordered path from $v$ to $v'$ with the edge labelled by $s$ joining $v'$ and $\varphi(v')$.  Since $c_s$ is the largest element of $\textbf{c}$, this is again an ordered path.  Hence $\mathscr G(\textbf{c})$ is an $S$-graph.
  \end {proof}

\subsection{}\label{7.6}

As one might expect establishing $P_3$ is rather more delicate.  We give two proofs. \textit{The first is direct and due to P. Lamprou.}  The second is more roundabout and is given in \ref {8.3}.  The first is rather subtle revealing a rather unexpected property of the graphs described below.

 Let $\mathscr G$ be a connected graph satisfying $P_1,P_2$ and take $v,v' \in V(\mathscr G)$. Let $v=v_1 \rightarrow v_2 \rightarrow \ldots \rightarrow v_n=v'$ be a path $p$ (not necessarily ordered) from $v$ to $v'$ and set
$$S_p(v,v'):= \sum_{j=1}^{n-1} c_{i,i+1}(r^{v_{i+1}}-r^{v_i}).$$

Of course $P_3$ is the assertion that $S_p(v,v')$ is independent of the path $p$.

\

$(P_{10})$.  Suppose $i_v=i_{v'}$.  Then $S_p(v,v')$ does not depend on $r^{i_v}$.

\

Call a path $v=v_1 \rightarrow v_2 \rightarrow \ldots \rightarrow v_n=v'$ bivalent of value $k \in \hat{T}$, if $k=i_v=i_{v'}$ and $i_{v_i}\neq k:i=2,3,\ldots, n-1$.

\subsubsection{}\label{7.6.1}

\begin {lemma}  Let $\mathscr G$ be a graph.  Then

\

(i).  $\mathscr G$ satisfies $P_{10}$ if it does so for all bivalent paths.

\

(ii).  For a bivalent path $v=v_1 \rightarrow v_2 \rightarrow \ldots \rightarrow v_n=v'$ of value $k$ the coefficient of $r^k$ in $S_p(v,v')$ equals $c_{n-1,n}-c_{1,2}$.
\end {lemma}

\begin {proof} (i) follows by concatenation of paths. (ii) by direct computation.
\end {proof}

\textbf{Remark}.  Notice that by (ii), conditions $P_2,P_3, P_{10}$ exclude a bivalent path being closed (that is $v_1=v_n$) unless it is the trivial path.  Again (ii) fails for paths which are not bivalent (even when the indices on end points coincide).

\subsubsection{}\label{7.6.2}

Recall the construction of \ref {7.2}. We say that a path $v_1 \rightarrow v_2 \rightarrow \ldots \rightarrow v_n$ crosses at $m$ if $\{i_{v_m},i_{v_{m+1}}\}=\{s,s+1\}$.

\begin {lemma}  Suppose $\mathscr G(\textbf{c}^-)$ satisfies $P_3,P_{10}$.  Then $\mathscr G(\textbf{c})$ satisfies $P_3$.
\end {lemma}

\begin {proof} Let $v_1 \rightarrow v_2 \rightarrow \ldots \rightarrow v_n$ be a closed path $p$ in $\mathscr G(\textbf{c})$.  To establish $P_3$, we must show that $S_p(v_1,v_n)=0$. If $p$ lies entirely inside $\mathscr G^+$ or $\mathscr G^-$, the assertion follows from $P_3$ for $\mathscr G(\textbf{c}^-)$.  Otherwise it must cross at $m_1,m_2,\ldots,m_\ell$ with $\ell$ even. Moreover we can assume that $v_{m_i}$ lies in $\mathscr G^-$ (resp. $\mathscr G^+$) for $i$ odd (resp. $i$ even).  Then the contributions to $S_p(v_1,v_n)$  at the crossings equal $c_s(r^s-r^{s+1})$, up to a sign which alternates with $i$, hence cancel.  On the other hand by $P_{10}$ for $\mathscr G(\textbf{c}^-)$, we can replace in the contribution to the sum coming from the component $p_i$ of $p$ joining $v_{m_{2i}+1}$ to $v_{m_{2i+1}}$, each vertex $v'' \in \mathscr G^+$ by $\varphi(v'') \in \mathscr G^-$.  Then the sum becomes that of a closed path in $\mathscr G^-$ and so is zero by $P_3$ for $\mathscr G(\textbf{c}^-)$.
\end {proof}

\subsubsection{}\label{7.6.2}

\begin {lemma}  Suppose $\mathscr G(\textbf{c}^-)$ satisfies $P_3,P_{10}$.  Then $\mathscr G(\textbf{c})$ satisfies $P_{10}$.
\end {lemma}

\begin {proof}  By Lemma \ref {7.6.1} we can assume that the path $p$ from $v$ to $v'$ is bivalent of value $k$ and non-trivial.  Then for any component of the sum lying in $\mathscr G^+$ we can replace each vertex $v'' \in \mathscr G^+$ by $\varphi(v'') \in \mathscr G^-$.  Then we obtain a new path $p'$ from $v$ to $v'$ lying entirely in $\mathscr G^-$.  In this case the vertices carrying the index $s+1$ are removed whilst the indices on the remaining vertices are unchanged.  In particular $p'$ is bivalent and non-trivial. Again the edges with index $s$ are correspondingly removed.  Yet unlike the situation described in the previous lemma, we \textit{need not} have $S_p(v,v')=S_{p'}(v,v')$.   However the indices over the remaining edges remain the same by virtue of \ref {7.2}$(*)$.  Now $\mathscr G^-$ is, up to a change of labelling defined by $\Psi^-$, isomorphic to $\mathscr G(\textbf{c}^-)$.  Hence $P_{10}$ for the latter implies by Lemma \ref {7.6.1} (and the Remark) that indices on the end vertices of $p'$ are both equal to some $i \in T \setminus \{s\}$.

We conclude that the indices on the end vertices in $p$ are either both equal to $i$ above or both equal to $s$.  Since $p$ is bivalent, the assertion follows from Lemma \ref {7.6.1}.
\end {proof}

\subsection{}\label{7.7}

\begin {cor} $\mathscr G(\textbf{c})$ is an $S$ graph.
\end {cor}

\subsection{Positivity and $P_9$}\label{7.8}

Fix a total order $\textbf{c}$ and recall \ref {4.2} the notion of positivity in the sector defined by $\textbf{c}$.  Since $\mathscr G(\textbf{c})$ is an $S$ graph we can assign to each $v \in V(\mathscr G(\textbf{c}))$ a function $f_v \in \hat{H}$.  Let $f^*_v$ be the function defined by viewing $v$ as an element of $V(\mathscr G(\textbf{c})^*)$.

 Recall the graph isomorphism $\varphi: \mathscr G^+ \rightarrow \mathscr G^-$.  Set $V^\pm=V(\mathscr G^\pm)$.  One has $\varphi(V^+)=V^-$.  Set $V^+_{s+1}:=\{v \in V^+|i_v=s+1, V^-_{s}:=\{v \in V^-|i_v=s\}$.  Clearly $\varphi(V^+_{s+1})=V^-_s$.

\begin {lemma}

\

(i)  For all $v \in V^+_{s+1}$ one has
$$f_{\varphi(v)}-f_v=c_s(r^s-r^{s+1}).$$

\

(ii) For all $v \in V^+ \setminus V^+_{s+1}$ there exists $c' \in \textbf{c}^-$ such that
$$f_{\varphi(v)}-f_v=(c_s-c')(r^s-r^{s+1}).$$

\

(iii) $f_v$ and $f^*_v$ satisfy positivity for all $v \in \mathscr G(\textbf{c})$.

\

(iv)  $\mathscr G(\textbf{c})$  satisfies $P_9$.
\end {lemma}

\begin {proof}
Take $v \in V^+_{s+1}$.  Then by \ref {6.6}$(*)$, we obtain
$$f_{\varphi(v)}=f_v+c_{v,\varphi(v)}(r^{i_{\varphi(v)}}-r^{i_v})=c_s(r^{s}-r^{s+1}).$$

Assume $v \in V^+ \setminus V^+_{s+1}$.  Then there is a path   $v=v_1 \rightarrow v_2 \rightarrow \ldots \rightarrow v_n$ with just the last element $v_n \in V^+_{s+1}$  in $\mathscr G^+$.   Concatenate this with the path $v_n \rightarrow \varphi(v_n) \rightarrow \varphi(v_{n-1})\rightarrow \ldots \rightarrow \varphi(v_1)$.

Through the formulae in \ref {7.2}$(3,4)$, we obtain
$$f_{\varphi(v)}-f_v=f_{\varphi(v_{n-1})}-f_{v_{n-1}}= (c_s-c_{v_{n-1},v_n})(r^s-r^{s+1}).$$

Thus (ii) results with $c'= c_{v_{n-1},v_n}$.  Of course the assertion that this result is independent of the choice of path is a consequence of $P_3$.

Recall that $c_s>c'$, for all $c' \in \textbf{c}^-$.  Then (iii) results from (i), (ii), the corresponding property for $\mathscr G(\textbf{c}^-)$ and duality.

For (iv) recall the notion of support of a function (\ref {4.6}).  Since no edge of  $\mathscr G^+$ has label $s$, it is immediate that $s\notin \Supp f_v$, for all $v \in \mathscr G^+$.  On the other hand by (ii) one has $s\in \Supp f_v$, for all $v \in \mathscr G^-$.   Consequently $P_9$ obtains by induction on $t$.

\end {proof}

\subsection{Trees}\label{7.9}

Recall \ref {7.2}.  We may also construct $S$-graphs inductively by taking $c_s$ to be the unique minimal element of $\textbf{c}$.  Paradoxically we get just one tree for every value of $t$.  If we include duality at each step then we can obtain $2^t$ trees.  As labelled graphs they are pairwise distinct.  Removing the labelling they become all isomorphic,  The construction is described briefly below.

As before set $\textbf{c}^-=\textbf{c}\setminus \{c_s\}$, but now with $c_s$ being the unique minimal element of $\textbf{c}$.  Assume that the $S$-graph $\mathscr G(\textbf{c}^-)$ has been determined.  Increase the index on every vertex and on every edge of $\mathscr G(\textbf{c}^-)$ by $1$.  Then to every vertex of $\mathscr G(\textbf{c}^-)$ adjoin an edge with label $1$ equipped with a second vertex having label $1$.  It is easy to check that the resulting graph $\mathscr G(\textbf{c})$ is an $S$-graph.  Moreover the new triads take the form $(a,b,c,d)$ with $i_a=i_d=1, i_{a,b}=i_{c,d}=1, i_{b,c}=j>1$.  In addition \textit{every} $j \in T\setminus \{1\}$ so obtains.  This every linear order on $\textbf{c}^-$ gives just one linear order on $\textbf{c}$ in contrast to the case when $c_s$ is the unique maximal element.

Applying this construction inductively starting at $t=1$ we obtain a tree with $2^t$ vertices in which every vertex with label $j \in T$ has $2^{t-j}$ vertices, as well as a single vertex with label $t+1$.  In this ``increasing coefficient'' case the corresponding tableaux (excepting that for $h$) has height $1$.  Moreover this construction is easy to describe in terms of tableaux.

At each step in the above construction we may also dualize.  For example dualizing at the last step corresponds to the ``decreasing coefficient case".  This leads to $2^t$ graphs which are pairwise distinct because they correspond to different linear orders and therefore (see above) by pairwise distinct sets of triads.  However as unlabelled graphs it is immediate that they are all isomorphic to a single specific tree.

\subsection{Octagons}\label{7.10}

One may be surprised that not all the (unlabelled) graphs $G(\textbf{c})$ are trees.  That one cannot simply permute indices on edges to translate these graphs amongst themselves originates from the fact that condition $P_4$ has to hold.  Moreover one cannot corresponding permute indices on vertices to compensate for the permutation of indices on edges because the indexing sets namely $T$ and $\hat{T}$ differ in cardinality by $1$.  Seemingly only duality (which one may prefer to call inversion) can act on both of these sets.

Octagons first appear when $|T|=3$.   In this case there are $4$ trees provided by the previous construction. The set of vertices is of cardinality $2^3=8$ breaking into subsets specified by the indices on vertices of cardinalities $(4,2,1,1)$.  The remaining two graphs are isomorphic and self-dual.  The common graph is an octagon and the subsets of vertices having a fixed index has $(2,2,2,2)$ as its set of cardinalities.

From the construction (see \ref {7.2}) it is immediate that the cardinality of a vertex in $\mathscr G(\textbf{c})$ having a fixed index is always a power of $2$ and moreover this cardinality can be calculated from the linear order $\textbf {c}$.  Of course as one runs over the indices on vertices this set of powers of $2$ must sum to $2^{|T|}$.  More interestingly it seems likely that this set of powers of two determines the isomorphism class of the unlabelled graph.

Since unlabelled graphs are not isomorphic we see another reason why the  $\mathscr G(\textbf{c})$ cannot be permuted by the symmetric group $S_t$.  However to each isomorphism class of unlabelled graphs one may a subset of $S_t$ by assigning the identity element of $S_t$ to the increasing coefficient case. Of course this is hardly canonical and indeed we have not been able to make much sense of the result.

\section{Relation to Tableaux}\label{8}

\subsection{}\label{8.1}

In this section we return to tableaux.  Their interest for the structure of $B(\infty)$ was discussed in the last part of \ref {4.2}.
%Concerning the structure of $B(\infty)$ it is possible that we could dispense with tableaux although we have yet to establish $P_8$ for the $S$-graphs we construct. Moreover this condition was one of the reasons why we considered tableaux in the first place and it seems unlikely that we could have arrived at $S$-graphs without them.
Again tableaux have the advantage that one can easily read off a function from its (complete) tableaux as well as the sector to which it belongs, rather than having to know the entire $S$-graph whose structure is hard to discern.

 A further interest in tableaux does not directly concern $B(\infty)$.  In a subsequent paper we show that the set $H^t$ of equivalence classes of tableaux with $t$ columns is a ``Catalan set'', that is to say $|H^t|=C(t)$, where $C(t)=\frac{(2t)!}{t!(t+1)!}$ is the $t^{th}$ Catalan number.

 There has been a considerable interest in Catalan sets and indeed a huge number of Catalan sets is known.  Yet $H^t$ lacks the left-right symmetry which is common for most Catalan sets. The closest we found to our present example was provided by the set of ideals $\mathscr I^t$ of strictly upper triangular $t\times t$ matrices stable under the conjugation by the diagonal matrices.  It is a Catalan set.  Yet already for $t=3$ the graph $G_{H^t}$ of links of $H^t$ is different to the graph $G_{\mathscr I^t}$ of inclusion of ideals.

More importantly every linear order $\textbf{c}$ on the $t-1$ element set $\{1,2,\ldots,t-1\}$ gives rise to a connected subgraph $G_{H^t(\textbf{c})}$ of $G_{H^t}$. These are not quite all distinct so we obtain (a little less than) $(t-1)!$ distinct isomorphism classes of labelled graphs.  It follows from Theorem \ref {8.5} that the cardinality of $H^t(\textbf{c})$ is always $2^{t-1}$ and their common intersection has cardinality $t$.  The nearest comparison that we found to this arises when one considers the decomposition of $\mathscr I^t$ into commutative ideals and into nilradicals of parabolics, both subsets having cardinality $2^{t-1}$ with common intersection of cardinality $t$. However even of for $t=3$ when this decomposition might seem exactly the same (because $(t-1)!$ happens to be $2$ in this case) the subgraph of nilradicals is not a \textit{connected} subgraph of $G_{\mathscr I^t}$, that is $P_5$ fails.  For $t=4$ the difference between $G_{H^t}$ and $G_{\mathscr I^t}$ is even greater - the former looks like a four-armed octopus and the latter an advertisement for ``dropbox''.   

%has maximal chain length $6$ of tableaux with decreasing numbers of blocks, whereas for the inclusion order of ideals ideals the maximal chain length is $7$.

Of course we would have liked to have obtained a relationship between $H^t$ and $\mathscr I^t$ since the latter has an elegant description obtained through the classical unpublished work of Peterson (see \cite {K}, \cite {PP1} and \cite {PP2}) using the affinisation of the symmetric group $S_{t-1}$.  At present we lean more towards the opinion that no such connection exists.

We shall need the following almost trivial result.

Fix $s \in T$ and let $H(\textbf{c}\setminus \{c_s\})$ denote the set of all $\{f \in H(\textbf{c})| c_s \notin \Supp (f-h)\}$.  Now suppose $c_s$ to be the unique largest element of $\textbf{c}$ and adopt the notation of \ref {7.2}.

\begin {lemma} $(\psi_s \times \psi^s)H(\textbf{c}^-_s)=H(\textbf{c}\setminus \{c_s\})$.
\end {lemma}

\begin {proof} Given a tableau $\mathscr T^-$ representing an element $f^- \in H(\textbf{c}_s^-)$ (in which by \ref {3.1} the rule \ref {7.2}$(1)$ applies) we may obtain a tableau $\mathscr T$ representing an element $f \in H$ by creating an empty column $C_s$ with $x$ co-ordinate $s$ and viewing the columns of $\mathscr T$ with $x$-co-ordinate $i\geq s$ as now having $x$ co-ordinate $i+1$. Notice that there are new relations introduced on passing to $P(\mathscr T)$ and these are of the form $c_i <c_s:i <s$.  Since $c_s$ was chosen to be the unique maximal element of $\textbf{c}$, we still obtain $\mathscr T \in H(\textbf{c})$, except it has an empty column with $x$ co-ordinate $s$.  Conversely given $\mathscr T \in H(\textbf{c})$ with an empty column with $x$ co-ordinate $s$ it may be collapsed to an element of $H(\textbf{c}^-_s)$.

 On the other hand $f$ is obtained from $f^-$ by replacing $r^i:i\geq s$ by $r^{i+1}$ and $c_i:i \geq s$ by $c_{i+1}$. These are just the maps $\psi^s$ and $\psi_s$ defined in \ref {7.2}$(2)$.

  Finally apply Lemma \ref {4.6} recalling \ref {2.3.7}.
\end {proof}

\textbf{Remark}.  Of course we do not need to be so pedantic over this relabelling, but there is a delicate point concerning the additional relations which are introduced.

\subsection{}\label{8.2}

Recall \ref {3.3.3} and \ref {5.3} in which $\mathscr G_H$ is defined. Fix a linear order $\textbf {c}$ and let $G_{H(\textbf{c})}$ be the subgraph of $\mathscr G_H$ whose vertices are the elements of $H(\textbf{c})$. Let $c_s$ be the unique maximal element of $\textbf{c}$.

Let $\mathscr G, \mathscr G'$ be graphs.  We define a graph embedding $\theta: \mathscr G \hookrightarrow \mathscr G'$ to be an injective map $\theta: V(\mathscr G )\hookrightarrow V(\mathscr G')$ such that $(\theta(v),\theta(v'))$ is an edge of  $\mathscr G'$ whenever $(v,v')$ is an edge of $\mathscr G$.

 \begin {prop} There exists a graph embedding $\theta:\mathscr G(\textbf{c})\hookrightarrow \mathscr G_{H(\textbf{c})}$ such that
 $$i_{\theta(v)}=i_v, \quad i_{\theta(v), \theta(v')}=i_{v,v'}, \forall v,v' \in \mathscr G(\textbf{c}). \eqno {(*)}$$

 In particular $P_9$ holds for $\mathscr G(\textbf{c})$.
 \end {prop}

 \begin {proof} The proof is by induction on $t:= |\textbf{c}|$.  It can immediately be checked when $t=1$. The induction step uses binary fusion. Thus recall the notation and hypotheses of \ref {7.2} and assume the assertion holds for $\textbf{c}$ replaced by $\textbf{c}^-$.

 From the definition of $\mathscr G^+:= \Psi^+_s(\mathscr G(\textbf{c}^-))$, and taking account (see proof of Lemma \ref {8.1}) that the additional relations $c_i <c_s:i <s$ are satisfied, it follows from the induction hypothesis that we have an embedding $\theta$ of $\mathscr G^+$ into $G_{H(\textbf{c})}$ satisfying $(*)$.

 A similar result may be obtained for $\mathscr G^-$ through duality though it is a little more subtle.  In the sense of \ref {6.8} we claim that
 $$(\Psi^+_{t+1-s}(\mathscr G(\textbf{c}^-)^*))^*= \Psi^-_s(\mathscr G(\textbf{c}^-))=:\mathscr G^-.\eqno{(**)}$$

  Since $\mathscr G^\pm$ are both isomorphic as unlabelled graphs to $\mathscr G(\textbf{c}^-)$ and duality is also an isomorphism of unlabelled graphs, it is enough to verify that labels on vertices and edges coincide in the two sides.

  Consider first edges. In the left hand side we apply successively three operations, first duality with respect to $T\setminus \{t\}$ which is $j \mapsto {t-j}$.   The second given by $\psi_{t+1-s}$ of \ref {7.2}
 which can be written as
 $$t+1-(j+1)\mapsto \left\{\begin{array}{ll}t+1-(j+1)& :t+1-(j+1)<t+1-s, \\
   t+2-(j+1)&:t+1-(j+1)\geq t+1-s.\\

\end{array}\right.$$

The third is duality with respect to $T$ which can be written as $t+1-j \mapsto j$.

Combined these give

 $$j\mapsto \left\{\begin{array}{ll}j& :j<s, \\
   j+1&:j \geq s,\\

\end{array}\right.$$
which is just $\psi_s$, as required.

  A similar computation applies to vertices. In the left hand side we apply successively three operations, first duality with respect to $\hat{T}\setminus \{t+1\}$ which is $j \mapsto t+1-j$, the second given by $\psi^s$ in \ref {7.2} which can be written as

 $$t+1-j\mapsto \left\{\begin{array}{ll}t+1-j& :t+1-j<t+1-s, \\
   t+2-j&:t+1-j\geq t+1-s,\\

\end{array}\right.$$
and the third is duality on $\hat{T}$ which can be written as $t+2-j \mapsto j$.

Combined these give
$$j \mapsto \left\{\begin{array}{ll}j& :j < s+1, \\
   j+1&:j\geq s+1.\\

\end{array}\right. $$
which is just $\psi^{s+1}$, as required.

By the induction hypothesis we have an embedding of $\mathscr G(\textbf{c}^-)$ into $G_{H(\textbf{c}^-)}$ which dualizes to an embedding of $\mathscr G(\textbf{c}^-)^*$ into $G_{H(\textbf{c}^-)^*}$.  Take $\mathscr T^- \in H(\textbf{c}^-)^*$.  Then as above the operation $\Psi^+_{t+1-s}$ translates to introducing an empty column with $x$ co-ordinate $t+1-s$ in $\mathscr T^-$ and the new relations so introduced take the form $c_{t+1-i}<c_{t+1-s}$ for $i>s$. The last operation of taking duals, translate these to $c_i<c_s$ for all $i>s$. Since $c_s$ was chosen to be the unique maximal element of $\textbf{c}$, the resulting tableau lies in  $ H(\textbf{c})$.
 Thus we obtain an embedding of $\theta$ of $\mathscr G^-$ into $\mathscr G_{H(\textbf{c})}$ satisfying $(*)$.

 Let us describe the edges $v \rightarrow \varphi(v): v \in V^+_{s+1}$ indexed by $c_s$ constructed in \ref {7.2} in terms of operations on the corresponding tableaux.

For each vertex $v \in \mathscr G_{H_s(\textbf{c}^-)}$, let $\mathscr T_v^-$ be a complete tableaux representing the corresponding function $f_v^- \in H_{s}(\textbf{c}^-)$.
 %Let $T_v^-$  be a complete tableaux representing $f_v$.
  Let $\mathscr T_v$ be the tableau representing the function $f_v \in H(\textbf{c})$ obtained as in the first part.
  %from $f_v^-$ by the relabelling of indices given by $\psi^-$ in \ref {7.2}.

 Let $C_j: j \in \hat{T}$ be the columns of $\mathscr T_v$.  Since the strongly extremal column in $\mathscr T^-_v$ has $x$ co-ordinate $s$, the strongly extremal column of $\mathscr T_v$ is $C_{s+1}$.

   The height of $C_{s+1}$ is the height $u$ of $\mathscr T_v$. Moreover if $u$ is odd (resp. even) then $C_{s+1}$ has no neighbour to its left (resp. right) at level $u$.  Again since $\mathscr T_v^-$ is complete it follows from Lemma \ref {2.3.3} that  the height of $C_{s-1}$ is $\geq u-2$ (resp. $\geq u-1$).

   By construction $C_s$ is the empty column in $\mathscr T_v$.  Then we may adjoin $[\frac{u}{2}]$ vertical left dominoes to $C_s$ to bring it to height $u-1$ (resp. $u$) which is at most one greater than the height of $C_{s-1}$.  We conclude that the resulting tableau is the completion $\hat{\mathscr T}_v$ of $\mathscr T_v$.  Then we may add a single block $B$ to $C_s$ (resp. and a single block to $C_{s+1}$) to obtain a new tableau $\mathscr T_{\varphi(v)}$ corresponding to an element $f_{\varphi(v)} \in H_s$. Notably this new tableau always has odd height and is complete. Thus the entry in $B$ is $c_s$.  Since the $\mathscr T_{\varphi(v)}$ has odd height and both $C_s,C_{s+1}$ both have the height of $\mathscr T_{\varphi(v)}$, no new relations are introduced.  Furthermore  $f_{\varphi(v)}-f_v=c_s(r^s-r^{s+1})$.

   (Notice we may remove vertical right dominoes from $\mathscr T_{\varphi(v)}$ to reduce its column with $x$ co-ordinate $s+1$ to a single block. Then the passage from $\mathscr T_v^-$ to $\mathscr T_{\varphi(v)}$ given just above is creating a column of height $1$ at the $(s+1)^{th}$ place.  In this we can suppose that every column of $\mathscr T_{\varphi(v)}$ is non-empty ( by adding two full rows if necessary).  Then the passage from $\mathscr T_v^-$ to $\mathscr T_{\varphi(v)}$ by implementing $(**)$ gives exactly the same result.)

   This shows that the edges with labels $c_s$ linking $\mathscr G^\pm$ in $\mathscr G(\textbf{c})$, defined by our previous rather ad hoc manner in \ref {7.2}, arise from the labelled graph embeddings $\theta: \mathscr G^+, \mathscr G^- \hookrightarrow G_{H(\textbf{c})}$.

   This last part of the proposition follows from \ref {2.4} and \ref {5.4}.

   \end {proof}

   \subsection{}\label{8.3}

   We can now give a second proof that

   \begin {cor} $P_3$ holds for $\mathscr G(\textbf{c})$.
   \end {cor}

   \begin {proof} Since $G_H$ is an evaluation graph, then a fortiori so is any subgraph (even if some edges may be missing).  Thus $P_3$ holds for $\mathscr G(\textbf{c})$ through the proposition.
   \end {proof}

   \subsection{}\label{8.4}

   A graph embedding $\theta :\mathscr G \hookrightarrow \mathscr G'$ is said to be strict if every edge between $(\theta(v),\theta (v'))$ in $\mathscr G'$  is the image of an edge $(v,v')$ in $\mathscr G$.

   \begin {lemma}

   \

   (i)   For all $v \in V(\mathscr G^+)$ the coefficient of $c_sm^j$ in $f_v$ is $-1$ if $j=s$ and \  $0$ otherwise.

   \

   (ii)   For all $v \in V(\mathscr G^-)$ the coefficient of $c_sm^j$ in $f_v$ is $-1$ if $j=s+1$ and \ $0$ otherwise.

   \

   (iii)  The graph embedding $\theta:\mathscr G(\textbf{c})\hookrightarrow \mathscr G_{H(\textbf{c})}$ is strict.
   \end {lemma}

   \begin {proof}  Since both $\mathscr G(\textbf{c})$ and its image satisfy $P_3,P_4,P_5$, we can identify the function $f_v$ defined by an element $v \in V(\mathscr G(\textbf{c}))$ by that defined by $\theta(v)\in \mathscr G_{H(\textbf{c})}$ of $v$.  Moreover in this $f_v$ is exactly the function attached to the image $\theta(v)\in \mathscr G_{H(\textbf{c})}$ of $v$.

   View $f_v$ as being linear in the $m^i:i\in \hat{T}$ and in the $c_i:i\in T$.  Since there is no edge with label $c_s$ in $\mathscr G^+$, it follows that the coefficient of $c_sm_j$ equals that of $h$.  Hence (i). Then (ii) follows from (i) by Lemma \ref {7.8}(ii).

   %namely $-1$ in $\mathscr G^+$ if $j=s$ and $0$ otherwise.

   The proof of (iii) is by induction on $|T|=t$.  Through the induction hypothesis we can assume that the restrictions of $\theta$ to $\mathscr G^\pm$ are strict.  Identifying $\mathscr G^\pm$ with their images it suffices to consider an edge between an element of $\mathscr G^+$ and an element of $\mathscr G^-$.  In this we further identify $\mathscr G^-$ with $\varphi(\mathscr G^+)$.

 It follows by (i), (ii) and the definition of a link in $\mathscr G_H$ that there can only be an edge between $f_v$ and $f_{\varphi(v')}:v,v' \in \mathscr G^+$ if its label is $s$.  In this case set $i_v=j,i_{\varphi(v')}=k$.  Then by Lemma \ref {5.5} we must have $f_{\varphi(v')}-f_v =c_s(r^k-r^j)$.  On the other hand the coefficient of $m^k$ in $f_{\varphi(v')}$ must be zero by the definition of $i_{\varphi(v')}$.  Since the coefficient of $c_sm^k$ in $f_v$ is not zero unless $k=s$, this forces $k=s$.  Yet there can be at most one edge emanating from $\varphi(v')$ with label $c_s$.  Moreover when $k=i_{\varphi(v')}=s$, we have seen that there is exactly one. This has $v'$ as its second vertex with $i_{v'}=s+1$ and comes from an edge in $\mathscr G(\textbf{c})$.

  %We conclude that there are no other edges in $\mathscr G(\textbf{c})$ arising from the embeddings $\theta: \mathscr G^+, \mathscr G^- \hookrightarrow G_{H(\textbf{c})}$.

 This completes the proof of the lemma.

\end {proof}

\subsection{}\label{8.5}

\begin {thm} $\theta$ is a labelled graph isomorphism of $\mathscr G(\textbf{c})$ onto $\mathscr G_{H(\textbf{c})}$.
   \end {thm}

\begin {proof}  The proof is by induction on $t$.  It is trivial for $t=1$. Recall the notation used in \ref {8.1}.

 By the induction hypothesis and Lemma \ref {8.1}, we can identify $\{f_{v'}: v' \in \mathscr G^+\}$ with the elements of $H(\textbf{c})$ not having $c_s$ in their support.  Thus if $v \in G_{H(\textbf{c})} \setminus \mathscr G^+$ is linked to $v' \in \mathscr G^+$, we must have $f_v=f_{v'}+c_s(r^{i_v}-r^{i_{v'}})$.

 Similarly by \ref {8.2}$(**)$ we can identify $\{f_{v'}: v' \in \mathscr G^{-*}\}$ with the elements of $H(\textbf{c})^*$ not having $c_{t+1-s}$ in their support. Thus if
 $v \in G_{H(\textbf{c})^*} \setminus (\mathscr G^{-})^*$ is linked to $v' \in (\mathscr G^{-})^*$, we must have
 $f_v=f_{v'}+c_{t+1-s}(r^{i_v}-r^{i_{v'}})$, so then  $f^*_v=f^*_{v'}+c_s(r^{t+2-i_v}-r^{t+2-i_{v'}})$

%We conclude that if $f_v=f_{v'}+c_{i_{v,v'}}(r^{i_v}-r^{i_{v'}})$ with $v' \in \mathscr G_s(\textbf{c}^-)$ and $v \in \mathscr G_{H(\textbf{c}^-)}$ then $v \in \mathscr G_s(\textbf{c}^-)$ unless $i_{v,v'}=s$.

Suppose that $\theta$ is not surjective and let $\mathscr G'$ be the complement of its image of $\theta$ in $\mathscr G_{H(\textbf{c})}$.  Since $\mathscr G_{H(\textbf{c})}$ is connected (Lemma \ref {6.5}) we can choose $v$ in the complement of the image of $\theta$  and $v'$ in the image such that
 $$f_v=f_{v'}+c_{i_{v,v'}}(r^{i_v}-r^{i_{v'}}), \eqno{(*)}$$
whilst by the above $i_{v,v'}=s$.

 Suppose $v' \in \mathscr G^+$.  By Lemma \ref {8.4}(i) and $(*)$ we must have $i_v=s$ to cancel the coefficient of $m^{i_v}$ in $f_v$.  This in turn means that the coefficient of $m_s$ in $f_{v'}$ is $-c_s$.  In view of $(*)$ this forces $i_{v'}>i_v=s$.

 Now let $\mathscr T_{v'}$ be a complete tableau representing $f_{v'}$. Let $u$ be the height of $\mathscr T_{v'}$.

 If $u$ is odd (resp. even) then by Lemma \ref {5.4} we deduce that $C_{i_{v'}}$ has height $u$ and is the leftmost (resp. rightmost) column of $\mathscr T_{v'}$ of height $u$.  Consequently we obtain a tableau $\mathscr T_v$ representing $f_v$ only when $u$ is odd and when we add a block in $R_u$  containing $c_s$ to the column $C_{i_v}=C_s$ of $\mathscr T_{v'}$ which has height $u-1$.

 Let $c_j$ be the entry in the $u-1$ row of $C_{i_{v'}}$. Take the completion $\hat{\mathscr T}_v$ of $\mathscr T_v$.  This moves $c_s$ to the top of the rightmost column of $\hat{\mathscr T}_v$ of height $u$ which is a left neighbour to $C_{i_{v'}}$ at level $u$.  Then by Lemma \ref {5.3} we obtain $c_s<c_j$.  On the other hand $\mathscr T_{v'}$ is complete, hence well-numbered, $i_{v'}>s$ and $u-1$ is even, so we obtain $j=i_{v'}-1$.  Since $c_s >c$, for all $c \in \textbf{c}^-$ we deduce that $j=s$ and so $i_{v'}=s+1$. Yet there is already an edge with label $s$ from a vertex $v' \in \mathscr G^+$ with label $s+1$.  Its second vertex $v$ lies in $\mathscr G^-$ which by $P_2$ gives a contradiction.

 Suppose $v' \in \mathscr G^-$.   By Lemma \ref {8.4}(ii) and $(*)$ we must have $i_v=s+1$ to cancel the coefficient of $m^{i_v}$ in $f_v$.  This in turn means that the coefficient of $m^{s+1}$ in $f_{v'}$ is $c_s$.  In view of $(*)$ this forces $i_{v'}<i_v=s$.

 Now let $\mathscr T_{v'}$ be a complete tableau representing $f_{v'}$. Let $u$ be the height of $\mathscr T_{v'}$.

 If $u$ is odd (resp. even) then by Lemma \ref {5.4} we deduce that $C_{i_{v'}}$ has height $u$ and is the leftmost (resp. rightmost) column of $\mathscr T_{v'}$ of height $u$.  Consequently we obtain a tableau representing $f_v$ only when $u$ is even and when we add a box containing $c_s$ to the column $C_{i_v}=C_{s+1}$ of $\mathscr T_{v'}$, which moreover must have height $u-1$.  Let $c_j$ be the entry in the $u-1$ row of $C_{i_{v'}}$.  Then exactly as in the previous case we obtain $c_s<c_j$, through completion and Lemma \ref {5.3}.  On the other hand $\mathscr T_{v'}$ is complete, hence well-numbered, $i_{v'}<s$ and $u-1$ is odd, so we obtain $j=i_{v'}$.  Since $c_s >c$, for all $c \in \textbf{c}^-$ we deduce that $j=s$ and so $i_{v'}=s$. Yet there is already an edge with label $s$ from a vertex $v' \in \mathscr G^-$ with label $s$.  Its second vertex $v$ lies in $\mathscr G^+$ which by $P_2$ gives a contradiction.

\end {proof}

\subsection{The Preparation Theorem}\label{8.6}

The Preparation theorem can be most fully expressed by saying that $\mathscr G_{H(\textbf{c})}$ (or $\mathscr G(\textbf{c})$) satisfies $P_1-P_{10}$, which of course follows from the two previous sections. However to bring out the most crucial property of our construction to light we state it as follows.

An ordered pair $(f,f'):f \in H^{t+1}_k(\textbf{c}), f' \in H^{t+1}_j(\textbf{c})$ is said to satisfy the condition $S$ if
 $$f'-f= \sum_{i=1}^{s-1}c_{u_i}(r^{v_i}-r^{v_{i+1}}),$$
 for some $u_1,u_2,\ldots,u_{s-1} \in T, v_1,v_2,\ldots,v_s \in \hat{T}$,
  with $v_i=j, v_s=k$  \textit{and} such that the $c_{u_i}$ are increasing.

  %The Master Theorem in the title of this paper is the following remarkable result which can be thought of expressing the invariance of $H^{t+1}(\textbf{c})$.

  \begin {thm} For all $f' \in H^{t+1}(\textbf{c})$ and all $k \in \{1,2,\ldots, t+1\}$ there exists $f \in H^{t+1}_k(\textbf{c})$ such that $(f,f')$ satisfies condition $S$.
  \end {thm}

\subsection {The Preparation Theorem for the Level One Case}\label{8.7}

 One may give a proof of the Preparation Theorem without using $S$-graphs when $f'$ has level $\leq 1$, that is to say when the corresponding tableau $\mathscr T'$ has only one row (possibly empty). For this we need the following construction which may help to understand the issues involved.

Fix $m \in \mathbb N^+$.  Choose a positive integer $s$ and a set $\textbf{k}_{m,s}$ of positive integers $1=k_1<k_2<\ldots <k_{s-1} \leq m-1$ by decreasing induction so that $k_{s-1}$ is maximal with the property that $c_{k_{s-1}-1} \leq c_{m-1}$ and that $k_{i-1}: i=s-1,s-2,\ldots,2$ is maximal with the property that $c_{k_{i-1}-1}\leq c_{k_{i}-1}$ (with the convention that $c_0=0$).  In this we shall say that $\textbf{k}_{m,s}$ is adapted to $\textbf{c}$ at distance $m$.  For example if $m=t+1$ and the $c_i$ are increasing then $s=t+1$ and $k_i:=i: i=1,2,\ldots,t$.

\begin {lemma}  Fix $u,u' \in \{1,2,\ldots,t+1\}$ distinct and $f' \in H_{u'}(\textbf{c})$.  If $f'$ has level $\leq 1$ then there exists $f \in H_u(\textbf{c})$ such that $f-f'$ satisfies condition $S$.
\end {lemma}

\begin {proof}  The assertion is easily verified for $h$ itself (for any choice of $\textbf{c}$).  Then by duality it holds when $\mathscr T'$ is a complete single row. Then by collapsing empty columns and induction on $t$ it suffices to consider the case when $u \notin \Supp f'$ and is maximal with this property in the leftmost connected set of empty columns of $\mathscr T'$.   Suppose $C'_m$ is the first empty column of $f'$ and $C'_n$ the first non-empty column to the right of the former (which exists by the first matching condition.) In particular $u=n-1$.

 One easily reduces to the case $m>1$.  Indeed suppose $m=1$.  In this case we obtain $f$ from $f'$ by adding one block to $C'_{n-1}$.  Then $f\in H_{n-1}(\textbf{c})$ and $f-f'=c_{n-1}(r^{n-1}-r^n)$.

Now suppose that $m>1$, so then $C'_1$ is non-empty and hence the strongly extremal column of $\mathscr T'$. Add a block to this column at level $2$.    Let $\textbf{k}_{m,s}$ be adapted to $\textbf{c}$ at distance $m$ and set $k_s=n$. Set $f_1=f'$ and define $f_{i+1}: i=1,2,3,\ldots,f_s$ inductively by adjoining to $f_i$ a block at level two to $C_{k_{i+1}}$ (which by choice of the $k_j$ has height one).  In this it is immediate that
$$f_{i+1}-f_i= c_{k_{i+1}-1}(r^{k_{i+1}}-r^{k_{i}}):i=1,2,\ldots,s-2, \quad f_{s}-f_{s-1}=c_{m-1}(r^{k_s}-r^{k_{s-1}}).\eqno{(*)}$$

Moreover given that $f_{i-1} \in H(\textbf{c}):s-1 \geq i\geq 2$ and noting that by construction $c_{k_i-1}\leq c_{k_{i-1}},c_{k_{i-1}+1},\ldots,c_{k_i-2}$, we conclude from \ref {3.2.1} that $f_{i} \in H(\textbf{c}):i=2,3,\ldots,s-1$.

Finally note that to the diagram of $f_{s}$ we may adjoin left vertical dominoes to the empty columns with $x$ co-ordinates $m,m+1,\ldots,n-1$ without changing $f_s$. Through \ref {5.3} and since already $f_{s-1} \in H(\textbf{c})$, it follows that $f_s \in H(\textbf{c})$. Finally let $f=f_{s+1}$ be given by the diagram obtained from that of $f_{s}$ by removing the box above the column with $x$ co-ordinate $n$. Then by Lemma \ref {5.4}, we conclude that  $f \in H_{n-1}(\textbf{c})$. Moreover
$$f_{s+1}-f_{s}=c_{n-1}(r^{n-1}-r^{k_s}).\eqno {(**)}$$

The condition that $f' \in H(\textbf{c})$ requires that $c_{m-1}\leq c_m, c_{m+1},\ldots,c_{n-1}$, in particular that $c_{m-1} \leq c_{n-1}$.  By construction the $c_{k_i-1}:i=2,\ldots,s-1$ are increasing and $c_{k_{s-1}-1} \leq c_{m-1}$.  We conclude through $(*)$ and $(**)$ above that $f-f'$ satisfies condition $S$.

\textbf{Remarks}.  The level one case has the advantage that we have less trouble with equivalence relations on tableaux.  However the main advantage (which is not so apparent from this example) is that we do not have to consider that $f$ may have had antecedents which already fix the ordered path we are trying to construct.  For example suppose we take the order relation $2<1<3<4$.  Then if one starts from the diagram $(1,1,1,0,1)$ and follows the above procedure, then we obtain the path $(1,1,1,0,1) \stackrel{2}{\rightarrow}(2,1,2,0,1)\stackrel{3}{\rightarrow}(2,1,2,0,2)=(2,1,2,2,2)\stackrel{4}{\rightarrow}(2,1,2,2,1)$, which is fine.  On the other hand suppose we start at $(2,1,2,2,2)$.  Then we may simply take the second and third steps in the above paths to obtain an element lying in $H_5$.  However if we did not know a priori that this path existed and start with the path try the path $(2,1,2,2,2)\stackrel{2}{\rightarrow}(1,1,1,0,1)$ which is not illogical considering that $2$ is the smallest element of $T$, then one cannot proceed to an element of $H_5$ by completing this first step to an ordered chain.  Briefly speaking the graph of links gives a needed global perspective to the possible ordered paths.

\end {proof}

\section{Indices}\label{9}

\subsection {Index of Notations}\label {9.1}

Symbols appearing frequently are given below in the paragraph they are first defined.

\ref {1.1}.  $\pi, W, e_\alpha, f_\alpha$.

\ref {1.2}.  $B(\infty)$.

\ref {1.9}.  $C_t$.

\ref {2.1} $T, \hat{T},\mathscr D, C_i,R_i,ht_\mathscr D$.

\ref {2.2} $\mathscr D^*$.

\ref {2.4} $H^{t+1},H,H_j$.

\ref {3.1} $\mathscr T, B(i,j),b(i,j),\varphi_j,\theta_j$.

\ref {3.2.1} $P(\mathscr T)$.

\ref {3.2.2} $L(\mathscr T)$.

\ref {3.3.3} $\mathscr G_H, i_v,i_{(v,v')}$.

\ref {4.2} $\textbf{c}, m^i,r^i, h, \mathscr T_h, H(\textbf{c})$.

\ref {4.4} $f_{R_s},f_\mathscr T$.

\ref {4.6} $\Supp$.

\ref {5.2} $P_s,f_{P_s}$.

\ref {6.1} $\mathscr G_H, \mathscr G_{H(\textbf{c})}, V(\mathscr G), E(\mathscr G)$.

\ref {6.2} $c_{v,v'}$.

\ref {7.1} $\mathscr G(\textbf{c})$.

\ref {7.2} $\textbf{c}^-, \Psi^\pm_s, \varphi$.

\subsection {Index of Notions}\label {9.2}

Notions used frequently are listed below where they first appear.

\ref {2.1} Diagrams, height function.

\ref {2.2} Extremal columns, boundary conditions, duality.

\ref {2.3} Dominoes, deplete diagrams, complete diagrams, strongly extremal column, half domino adjunction.

\ref {3.1} Tableaux, blocks, extremal blocks, well-numbered.

\ref {3.3}  Single block linkage, quasi-extremal columns, graph of links.

\ref {4.2} Sector.

\ref {4.3} Duality of functions.

\ref {6.3}  Pointed chain.

\ref {6.4}  Marked vertex.

\ref {6.6}  Triads.

\ref {6.7}  Ordered path, $S$-condition,  $S$-graphs.

\ref {7.2}  Binary fusion.

\end {document}